\documentclass[12pt]{amsart}

\usepackage{amsmath, amsthm, amssymb,amscd}
\usepackage{fullpage, paralist,enumerate}
\usepackage{verbatim}
\usepackage[all]{xy}
\usepackage[parfill]{parskip}
\usepackage{graphicx}

\usepackage{afterpage}

\usepackage{amsmath}
\usepackage{amsfonts}
\usepackage{mathcomp}
\usepackage{textcomp}
\usepackage{amssymb}
\usepackage{latexsym}
\usepackage{graphicx}
\usepackage{courier}
\usepackage[english]{babel}
\usepackage{mathbbol}
\usepackage{multirow}
\usepackage{latexsym}
\usepackage{epsfig}
\usepackage{subfigure}
\usepackage{dcolumn}
\usepackage{bm}
\usepackage{colordvi}
\usepackage{color}
\usepackage{booktabs}
\usepackage{stmaryrd}
\usepackage{cite}
\usepackage[linesnumbered,commentsnumbered,ruled]{algorithm2e}
\usepackage[noend]{algpseudocode}
\usepackage{epstopdf}
\input xy
\xyoption{all}
\newcommand{\commentout}[1]{}

\newtheorem{thm}{Theorem}[section]

\newtheorem{prop}[thm]{Proposition}
\newtheorem{cor}[thm]{Corollary}

\newtheorem{rmk}[thm]{Remark}

\newcommand{\nwc}{\newcommand*}

\nwc{\ben}{\begin{equation*}}
\nwc{\bea}{\begin{eqnarray}}
\nwc{\beq}{\begin{eqnarray}}
\nwc{\bean}{\begin{eqnarray*}}
\nwc{\beqn}{\begin{eqnarray*}}
\nwc{\beqast}{\begin{eqnarray*}}

\nwc{\eal}{\end{align}}
\nwc{\een}{\end{equation*}}
\nwc{\eea}{\end{eqnarray}}
\nwc{\eeq}{\end{eqnarray}}
\nwc{\eean}{\end{eqnarray*}}
\nwc{\eeqn}{\end{eqnarray*}}

\CompileMatrices

\nwc{\nn}{\nonumber}
\nwc{\mb}{\mathbf}
\nwc{\ml}{\mathcal}

\newcommand{\lt}{\left}
\newcommand{\rt}{\right}

\nwc{\vep}{\varepsilon}
\nwc{\ep}{\epsilon}
\nwc{\vrho}{\varrho}
\nwc{\orho}{\bar\varrho}
\nwc{\vpsi}{\varpsi}
\nwc{\lamb}{\lambda}
\nwc{\om}{\omega}
\nwc{\Om}{\Omega}
\nwc{\al}{\alpha}
\nwc{\sgn}{\mbox{\rm sgn}}

\nwc{\IA}{\mathbb{A}} 
\nwc{\bi}{\mathbf{i}}
\nwc{\ba}{\mathbf{a}}
\nwc{\bmb}{\mathbf{b}}
\nwc{\bo}{\mathbf{o}}
\nwc{\IB}{\mathbb{B}}
\nwc{\IC}{\mathbb{C}} 
\nwc{\ID}{\mathbb{D}} 
\nwc{\IM}{\mathbb{M}} 
\nwc{\IP}{\mathbb{P}} 
\nwc{\II}{\mathbb{I}} 
\nwc{\IE}{\mathbb{E}} 
\nwc{\IF}{\mathbb{F}} 
\nwc{\IG}{\mathbb{G}} 
\nwc{\IN}{\mathbb{N}} 
\nwc{\IQ}{\mathbb{Q}} 
\nwc{\IR}{\mathbb{R}} 
\nwc{\IT}{\mathbb{T}} 
\nwc{\IZ}{\mathbb{Z}} 

\nwc{\cE}{{\ml E}}
\nwc{\cP}{{\ml P}}
\nwc{\cQ}{{\ml Q}}
\nwc{\cL}{{\ml L}}
\nwc{\cX}{{\ml X}}
\nwc{\cW}{{\ml W}}
\nwc{\cZ}{{\ml Z}}
\nwc{\cR}{{\ml R}}
\nwc{\cV}{{\ml V}}
\nwc{\cT}{{\ml T}}
\nwc{\crV}{{\ml L}_{(\delta,\rho)}}
\nwc{\cC}{{\ml C}}
\nwc{\cO}{{\ml O}}
\nwc{\cA}{{\ml A}}
\nwc{\cK}{{\ml K}}
\nwc{\cB}{{\ml B}}
\nwc{\cD}{{\ml D}}
\nwc{\cF}{{\ml F}}
\nwc{\cS}{{\ml S}}
\nwc{\cM}{{\ml M}}
\nwc{\cG}{{\ml G}}
\nwc{\cH}{{\ml H}}
\nwc{\bk}{{\mb k}}
\nwc{\bn}{{\mb n}}
\nwc{\bp}{{\mb p}}
\nwc{\bz}{\mb z}
\nwc{\bl}{{\mb l}}
\nwc{\bj}{{\mb j}}
\nwc{\bs}{{\mb s}}
\nwc{\by}{\mathbf{h}}
\nwc{\bZ}{\mathbf{Z}}
\nwc{\bF}{\mathbf{F}}
\nwc{\bE}{\mathbf{E}}
\nwc{\bV}{\mathbf{V}}
\nwc{\bY}{\mathbf Y}
\nwc{\br}{\mb r}
\nwc{\pft}{\cF^{-1}_2}
\nwc{\bU}{{\mb U}}
\nwc{\bG}{{\mb G}}
\nwc{\bg}{\mathbf{g}}
\nwc{\mbf}{\mathbf{f}}
\nwc{\mbe}{\mathbf{e}}
\nwc{\be}{\mathbf{e}}
\nwc{\ind}{\operatorname{I}}
\nwc{\mbx}{\mathbf{f}}
\nwc{\bb}{\mathbf{g}}
\nwc{\xmax}{f_{\rm max}}
\nwc{\xmin}{f_{\rm min}}
\nwc{\suppx}{\hbox{\rm supp} (\mbf)}
\nwc{\cI}{\IZ^2_N}
\nwc{\chis}{{\chi^{\rm s}}}
\nwc{\chii}{{\chi^{\rm i}}}
\nwc{\pdfi}{{f^{\rm i}}}
\nwc{\pdfs}{{f^{\rm s}}}
\nwc{\pdfii}{{f_1^{\rm i}}}
\nwc{\pdfsi}{{f_1^{\rm s}}}
\nwc{\thetatil}{{\tilde\theta}}
\nwc{\red}{\color{red}}
\nwc{\blue}{\color{blue}}

\nwc{\prox}{\hbox{prox}}
\nwc{\diag}{\hbox{\rm diag}}
\nwc{\supp}{{\hbox{\rm supp}}}

\nwc{\sloc}{J_{\rm f}}
\nwc{\bu}{{\mb u}}
\nwc{\bv}{{\mb v}}
\nwc{\cU}{\mathcal{U}}
\nwc{\cN}{\mathcal{N}}
\nwc{\bN}{\mathbf{N}}
\nwc{\mbm}{\mathbf{m}}
\nwc{\bw}{\mathbf{w}}
\nwc{\bom}{\mathbf{w}}
\nwc{\bt}{\mathbf{t}}
\nwc{\z}{y}
\nwc{\cY}{\mathcal{Y}}
\nwc{\bM}{\mathbf{M}}
\nwc{\half}{{1\over 2}}
\nwc{\Sf}{S_{\rm f}}
\nwc{\Jf}{J_{\rm f}}
\nwc{\nul}{\hbox{\rm null}_\IR}
\nwc{\spanR}{\hbox{\rm span}_\IR}
\nwc{\Arg}{\hbox{\rm Arg~}}
\nwc{\fdr}{S_{\rm f}}
\nwc{\phase}[1]{\exp\lt[i\measured #1\rt]}

\nwc{\im}{{\rm i}}

\nwc{\cle}{\preccurlyeq}

\nwc{\lb}{\llbracket}
\nwc{\rb}{\rrbracket}
\nwc{\modpi}{{{\rm mod}\,2\pi}}
\nwc{\px}{P_X}
\nwc{\pxp}{P_X^\perp}
\nwc{\py}{P_Y}
\nwc{\rx}{R_X}
\nwc{\ry}{R_Y}
\nwc{\pxtil}{\tilde P_X}
\nwc{\pxptil}{\tilde P_X^\perp}
\nwc{\rxtil}{\tilde R_X}

\usepackage{lipsum}
\usepackage{amsfonts}
\usepackage{graphicx}
\usepackage{epstopdf}

\usepackage{enumitem}
\setlist[enumerate]{leftmargin=.5in}
\setlist[itemize]{leftmargin=.5in}

\title{Fixed Point Analysis of Douglas-Rachford Splitting for  Ptychography and Phase Retrieval}

\author{Albert Fannjiang$^*$
 \address{
Department of Mathematics, University of California, Davis, California  95616, USA.
} \and Zheqing Zhang
\address{ Department of Mathematics, University of California, Davis, California  95616, USA. } 
}
\thanks{$^*$Corresponding author: fannjiang@math.ucdavis.edu}
\thanks{\hspace{0.1cm} \blue SIAM Journal on Imaging Sciences (2020)}

\usepackage{amsopn}

\begin{document}

\title{Fixed Point Analysis of Douglas-Rachford Splitting for  Ptychography and Phase Retrieval}

\maketitle 
\begin{abstract}
Douglas-Rachford Splitting (DRS) methods based on the proximal point algorithms for the Poisson and Gaussian log-likelihood functions 
are proposed for ptychography and phase retrieval.

Fixed point analysis shows  that the DRS iterated sequences are always bounded explicitly in terms of the step size
and that the fixed points are   attracting if and only if the fixed points are regular solutions.
This alleviates  two major drawbacks of the classical Douglas-Rachford algorithm:
slow convergence when the feasibility problem is consistent and divergent behavior when the feasibility problem is inconsistent. 

Fixed point analysis also leads to a simple, explicit expression for the optimal step size in terms of the spectral gap
of an underlying matrix. 

When applied to the challenging problem of {\em blind} ptychography, which seeks to recover both the object and the probe simultaneously,
Alternating Minimization with the DRS inner loops, even with  a far from optimal step size,  converges geometrically under the nearly minimum conditions established in the uniqueness theory.

\commentout{Three boundary conditions are considered in the simulations: periodic, dark-field and bright-field boundary conditions.
The dark-field boundary condition is suited for isolated objects while the bright-field boundary condition is for non-isolated objects.
The periodic boundary condition is a mathematically convenient reference point. Depending on the availability of the boundary prior
the dark-field and 
the bright-field boundary conditions may or may not be enforced in the reconstruction.  Not surprisingly, enforcing the boundary condition 
improves the rate of convergence, sometimes in a significant way. Enforcing the bright-field condition in the reconstruction
can also remove
the linear phase ambiguity. 
}

\end{abstract}



\section{Introduction}

Phase retrieval may be posed as an inverse problem in which an object vector with certain properties is to be reconstructed from the intensities of its Fourier transform. By encoding the properties and the Fourier intensities as constraint sets, phase retrieval can be cast as a feasibility problem, i.e. the problem of finding a point in the intersection of the constraint sets.
The challenge  is that the intensities of the Fourier transform results in a non-convex constraint
set (a high dimensional torus of variable radii). 

Projection algorithms comprise a general class of iterative methods for solving feasibility problems by projecting onto each of the constraint sets at each step 
\cite{BB}. The most basic projection algorithm is von Neumann's Alternating Projections (AP) (aka Error Reduction in phase retrieval \cite{Fie82}). However, AP tends to stagnate when applied to phase retrieval,
resulting in poor performance. 

A better method   than AP is the classical Douglas-Rachford algorithm (a.k.a. Averaged Alternating Reflections (AAR)) \cite{DR, GM76, LM79, GM76}, which apparently can avoid the stagnation problem in many non-convex 
problems. When applied to phase retrieval, the classical Douglas-Rachford algorithm is a special case of Fienup's Hybrid-Input-Output algorithm \cite{BCL02,Fie82}.

In addition to the standard phase retrieval, AAR has been applied
to  ptychography under the name of difference map \cite{DM08,probe09,ADM}. 
Ptychography uses a localized coherent probe to illuminate different parts of a unknown extended object 
and collect multiple diffraction patterns 
 as measurement data  (Figure \ref{fig0}). The redundant information in the overlap between adjacent illuminated spots is then exploited  to improve phase retrieval methods  \cite{Pfeiffer,Nugent}. 
  Recently ptychography  has been extended to the Fourier domain \cite{FPM13, Yang14}.  In Fourier ptychography,  illumination angles are scanned sequentially with a programmable array source with the diffraction pattern measured at each angle  \cite{Tian15,FPM15}. Tilted illumination samples different regions of Fourier space, as in synthetic-aperture and structured-illumination imaging.

Local, linear convergence of AAR as applied to phase retrieval as well as  ptychography  was recently proved in  \cite{FDR,ptych-unique}. Conditions for global convergence, however, are not known. 
Numerical evidence points to sub-linear rate when convergence happens. 
On the other hand, for inconsistent feasibility problems,  AAR iteration  is known to diverge to infinity even in the convex case (see Proposition \ref{prop0}(ii)). This poses a great challenge to AAR when the data contain noise because in phase retrieval the dimension of the measurement data
is much higher than that of the unknown object (an over-determined system).

The purpose of this work is to develop reconstruction schemes based on more general Douglas-Rachford splitting (DRS) with  adjustable step sizes, perform the fixed point analysis and demonstrate numerical convergence. AAR is the limiting case of DRS. 

The DRS method is an optimization method based on proximal operators, a natural extension of projections,  and is closely related to the Alternating Direction Method of Multipliers (ADMM). 
The performance of DRS and ADMM in the non-convex setting depends sensitively on
the choice of the loss functions as well as the step sizes. 
Typically, global convergence of DRS requires a loss function possessing a uniformly Lipschitz gradient  and sufficiently large step sizes \cite{FDR, LP,Hesse}, both
of which, however,  tend to hinder the performance of DRS. 

In this paper, the loss functions are based on the log-likelihood function for the most important Poisson noise, which does not have
a uniformly Lipschitz gradient, with an optimal step size, which is necessarily quite large.  

We show by a fixed point analysis that the DRS method is well behaved
in the sense that the DRS iterated sequences are always bounded (explicitly in terms of the step size) 
and that the fixed points are   attracting if and only if the fixed points are regular solutions.
In other words, the DRS methods remove AAR's two major drawbacks: 
slow convergence when the feasibility problem is consistent and divergent behavior when the feasibility problem is inconsistent.

Moreover, the fixed point analysis leads to the determination of  the  optimal step size and, along with it, 
simple and efficient algorithms with no tuning parameter (Averaged Projection-Reflection). 
The main application considered is the more  challenging form of ptychography called {\em blind ptychography} which seeks to recover both the unknown object and
the probe function simultaneously.
When 
properly initialized, the DRS algorithms with the optimal step size converge globally and
geometrically to the true solution modulo the inherent ambiguities. 

  \begin{figure}[t]
\begin{center}
\includegraphics[width=10cm]{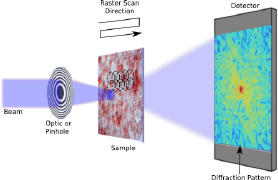}
\caption{Simplified ptychographic setup showing a Cartesian grid used for the overlapping raster scan positions \cite{parallel}. See Appendix \ref{app:matrix} for details. 
}
\label{fig0}
\end{center}
\end{figure}

The rest of the paper is organized as follows. In Section \ref{sec:DRS}, we introduce the Douglas-Rachford splitting method as the key ingredient of our reconstruction algorithms, Gaussian-DRS and Poisson-DRS. We give the fixed point and stability analysis in Sections \ref{sec:fixed}, \ref{sec:stab} and \ref{sec:gap}. In Section \ref{sec:stepsize}, we discuss the selection of
the optimal step size.  In Section \ref{sec:blind}, we discuss the application to blind ptychography and
 in Section \ref{sec:num}, we present numerical experiments. In Appendix \ref{app:matrix}, we discuss the structure of the measurement matrices.  In Appendix \ref{app:likelihood}, we show that Gaussian-DRS is an asymptotic form of Poisson-DRS. In Appendix \ref{app:Poisson} we give a perturbation analysis for the Poisson DRS.  In Appendix \ref{app:eigen}, we analyze the eigenstructure of the measurement matrices. We conclude  in Section \ref{sec:last}. A preliminary version of the present work is given in \cite{DRS}.

\section{Averaged Alternating Reflections (AAR)}
\label{sec:AAR} 

The classical Douglas-Rachford algorithm  is based on the following characterization of {\em convex} feasibility problems. 

Let $X$ and $Y$ be
the constraint sets. Let $P_X$ be the projection onto $X$ 
and $R_X=2P_X-I$ the corresponding reflector. $P_Y$ and $R_Y$ are defined likewise. Then 
\beq
\label{dr0}
u\in X\cap Y\quad \mbox{if and only if }\quad u=R_YR_X u 
\eeq
\cite{Boyd17}. The latter fixed point equation motivates the Peaceman-Rachford (PR) method: For $k=0,1, 2, \ldots$
\[
u_{k+1}= R_Y R_X y_{k}. 
\]
The classical Douglas-Rachford algorithm is the {\em averaged} version of PR: For $k=0,1, 2, \ldots$
\beq
\label{dr}
u_{k+1}= \half u_k+\half R_Y R_X u_{k},
\eeq
hence the name {\em Averaged Alternating Reflections} (AAR). 

A standard result for AAR in the convex case is this.

\begin{prop} \cite{BCL04} Suppose $X$ and $Y$ are closed and convex sets of a finite-dimensional vector space $E$. Let $\{u_k\}$ be an AAR-iterated sequence for any $u_1\in E$. Then one of the following alternatives holds:

(i) $X\cap Y\neq \emptyset$ and $(u_k)$ converges to a point $u$ such that $P_X u\in X\cap Y$;\\
(ii) $X\cap Y= \emptyset$ and $\|u_k\|\to\infty$. 
\label{prop0}
\end{prop}
In the consistent case  (i), the limit point $u$ is a fixed point of the AAR map \eqref{dr}, which after projection is in $X\cap Y$. However, 
the convergence rate of AAR is in general sublinear \cite{DR-rate2,DR-rate1}. 
The inconsistent case (ii) arises from noisy data or modeling errors resulting in divergent AAR iterated sequences, a major drawback of AAR  since the inconsistent case is prevalent with noisy data because of the higher dimension of data compared to the object. 

The AAR map \eqref{dr} is often written in the following form
\beq
\label{3step}
u_{k+1}= u_k+P_YR_Xu_k-P_Xu_k
\eeq
which is 
equivalent to the 3-step iteration
\beq
\label{dr2}
y_{k} &=& P_Xu_k; \\
z_{k} &=& P_Y (2y_{k}-u_k)=P_YR_X u_k\\
u_{k+1}&= & u_k+ z_{k}-y_{k}\label{dr2'}
\eeq

\subsection{Phase retrieval as feasibility}\label{sec:solution}

 For any finite dimensional vector $u$, define  its modulus vector $|u|$ as
 $
 |u|[j]= |u[j]|
 $
 and its phase vector $\sgn(u)$ as 
\[
\sgn(u)[j]=\lt\{\begin{matrix} 
1&\mbox{if $u[j]=0$}\\
u[j]/|u[j]|&\mbox{else. }
\end{matrix}\rt.
\]
where $j$ is the index for the vector component. Because of the value of $\sgn(u)$ where $u[j]=0$ is arbitrarily selected, such points are points of discontinuity of the $\sgn$ function. 

In phase retrieval including ptychography, we can write the given data $b$ as 
\beq
\label{pr}
b=|u|\quad \mbox{ with}\quad  u=A f
\eeq
for
some measurement matrix $A$ and unknown object $f$. For phase retrieval and ptychography,
$A$ has some special features described in Appendix \ref{app:matrix}. For most of the subsequent analysis, however,
these special features are not relevant.

Let  $O$ be the object space, typically a finite-dimensional complex vector space, and $X=AO$.
Since the object is a two dimensional, complex-valued image, we let 
$O=\IC^{n^2}$ where $n$ is the number of pixels in each dimension. 

Let $N$ be the total number of data. The data manifold 
\[
Y:=  \{u\in \IC^N: |u|=b\}
\]
 is a $N$ (real) dimensional torus.  
For phase retrieval it is necessary that  $N>2n^2$. Without loss of generality  we assume that $A$ has a full rank.

The problem of phase retrieval and ptychography can be formulated as the feasibility problem
\beq
 \label{feas}
\hbox{Find}\quad  u\in  X \cap Y, 
 \eeq
 in the transform domain $\IC^N$ instead of 
the object domain $\IC^{n^2}$.

Let us clarify  the meaning of solution in  the transform domain since $A$ is overdetermining. 
Let $\odot$ denotes the component-wise (Hadamard) product and we can write
\beq
\label{proj}
P_X u=AA^+ u,&& P_Y u=b\odot\sgn(u)\\
R_X=2P_X-I, &&R_Y=2P_Y-I \nn
\eeq
where $A^+ := (A^*A)^{-1}A^*$ is the  pseudo-inverse of $A$.

We refer to $u=e^{\im \alpha}Af, \alpha\in \IR$, as the {\em true} solution (in the transform domain),
up to a constant phase factor $e^{\im \alpha}$. 
We say that $u$ is a {\em generalized solution} (in the transform domain) if 
\[
|\hat u|=b,\quad \hat u:=P_X u. 
\]
Accordingly,  the alternative (i) in Proposition \ref{prop0} means
that if a convex feasibility problem is consistent then   every AAR iterated sequence converges to a generalized solution and hence every fixed point is a generalized solution.

Typically a generalized solution $u$  is neither a feasible solution (since $|u|$ may not equal $b$)  nor unique (since $A$ is overdetermining) and, 
if $P_X z=0$,  $u+z$ is also a generalized solution.  We call  $u$ a {\em regular} solution if $u$ is a  generalized solution and
$P_X u=u$.

Let $\hat u= P_X u$ for a generalized solution $u$. 
Since  $P_X \hat u=\hat u$ and $|\hat u|=b$, $\hat u$ is a regular solution.
Let us state this simple fact for easy reference.  
\begin{prop}
\label{prop1'} If $u$ is a generalized solution, 
then $P_X u$ is a regular solution. 
\end{prop}

The goal of the inverse problem \eqref{pr} is the unique determination of $f$, up to a constant phase factor, from the given data $b$. 
In other words,  uniqueness holds if, and only if, all regular solutions $\hat u$ have the form
\beq
\hat u=e^{\im \alpha} A f
\eeq
or equivalently, any generalized solution $u$ is an element of 
the $(2N-2n^2)$-dimensional manifold
 \beq
\label{gen}\label{unique}
\{e^{\im \alpha}Af-z: P_Xz=0,\,\,  z\in \IC^N,\,\,\alpha\in\IR\}.  
\eeq
In the transform domain, the uniqueness is
characterized by the uniqueness of the regular solution, up to a constant phase factor. 
Geometrically, uniqueness  means that the intersection $X\cap Y$ is a circle (parametrized $e^{\im \alpha}$ times $Af$).

As proved in  \cite{FDR}, when the uniqueness \eqref{unique} holds, the fixed point set of the AAR map \eqref{dr} is exactly the continuum set 
\beq
\label{fixed}
\{u=e^{\im \alpha}Af-z: P_Xz=0,\,\,\sgn(u)=\alpha+\sgn(Af), \,\,z\in \IC^N, \alpha\in \IR\}. 
\eeq
In \eqref{fixed}, the phase relation $\sgn(u)=\alpha+\sgn(Af)$ implies that
$z=\eta\odot \sgn(u),\eta\in \IR^N, b+\eta\ge 0.$ So the set \eqref{fixed} can be written as
\beq
\label{fixed2}
\{e^{\im \alpha} (b-\eta)\odot \sgn(Af):P_X(\eta\odot\sgn(Af))=0,\,\, b+\eta\ge 0,\,\,\eta\in \IR^N,\alpha\in \IR\}, 
\eeq
which is an $(N-2n^2)$ real-dimensional set,  a much larger set than the circle $\{e^{\im \alpha} Af: \alpha \in \IR\}$ for a given $f$. 
\commentout{
In \cite{FDM} the set in \eqref{fixed} is shown to be exactly the fixed point set of
 the Hybrid-Projection-Reflection family of maps which includes AAR as a special case 
\cite{HPR}. In turn,  the Hybrid-Projection-Reflection family is a subset of the Difference Map family which may possess a considerably larger set of fixed points than \eqref{fixed} \cite{Elser, FDM}. 
}
 On the other hand, the fixed point set \eqref{fixed2} is
 $N$-dimension lower than  
 the set  \eqref{gen} of generalized solutions. 
 
 A more intuitive characterization of the fixed points can be obtained by applying $R_X$ to the set \eqref{fixed2}. Since
 \[
 R_X [e^{\im \alpha} (b-\eta)\odot \sgn(Af)] = e^{\im \alpha} (b+\eta)\odot \sgn(Af)
 \]
 amounting to the sign change in front of $\eta$, 
the image set of \eqref{fixed2} under the map $R_X$ is
 \beq
\label{fixed2'}
\{e^{\im \alpha} (b+\eta)\odot \sgn(Af):P_X(\eta\odot\sgn(Af))=0,\,\, b+\eta\ge 0,\,\,\eta\in \IR^N,\alpha\in \IR\}. 
\eeq
The set \eqref{fixed2'} is the fixed point set  of the alternative form of AAR:
\beq
x_{k+1}&=& \half x_k+\half R_X R_Y x_k \label{aar2}
\eeq
in terms of  $x_k:=R_X u_k$. The expression \eqref{fixed2'} says that the fixed points of \eqref{aar2} are generalized solutions with the ``correct" Fourier phase.

However, the boundary points of the fixed point set \eqref{fixed2'} are degenerate in the sense that they have vanishing components, i.e. $(b+\eta)[j]=0$ for some $j$ and   can slow down convergence \cite{Fie86}.  Such points are points of discontinuity of the AAR map \eqref{aar2}
because they are points of discontinuity of $P_Y=b\odot\sgn(\cdot)$  (see also the comment below \eqref{P1}). 
 Indeed, even though AAR converges linearly  
in the vicinity of the true solution, numerical evidence suggests that
globally (starting with a random initial guess) AAR  converges sub-linearly  (cf. \cite{DR-rate2,DR-rate1}). 
Due to the non-uniformity of convergence, the additional step of applying $P_X$ (Proposition \ref{prop0}(i))   at the ``right timing" of the iterated process can jumpstart the geometric convergence regime  \cite{FDR}.

\section{Douglas-Rachford Splitting (DRS)} \label{sec:DRS} 
Douglas-Rachford Splitting (DRS) is an optimization method for solving 
the following minimization problem:
\begin{equation}
\min\limits_{u}K(u)+ L(u)
\label{DRS}
\end{equation}
where the loss functions $ L$ and $K$ represent the data constraint $Y$ and
the object constraint $X$, respectively. 

To deal with the divergence behavior of AAR (Proposition \ref{prop0} (ii)) in the case of, e.g.  noisy data, we consider the Poisson log-likelihood cost functions \cite{ML12,Poisson2}
\beq
\mbox{\rm Poisson:}\quad  L(u) &=&\sum_i |u[i]|^2-b^2[i]\ln |u[i]|^2\label{Poisson}
\eeq
based on the maximum likelihood principle for the Poisson noise model. The Poisson noise is the most prevalent noise in X-ray coherent diffraction. {  There is, however, a disadvantage of working with \eqref{Poisson}, i.e. it has a 
divergent derivative where $u(i)$ vanishes but $b(i)$ does not. This roughness can be softened by considering an asymptotic form }
\beq
\mbox{\rm Gaussian:}\quad   L(u)&=&\half \| |u|- b\|^2 \label{Gaussian}
\eeq
In Appendix \ref{app:likelihood}, we show that 
the Poisson log-likelihood function \eqref{Poisson}  is  asymptotically reducible  to \eqref{Gaussian}. 

With the constraint $u=Ag$, $g$ is a stationary point in the object domain if and only if 
\beqn
g&=&A^* \lt[\sgn(Ag)\odot b\rt]. 
\eeqn
In the noiseless case, $|Af|=b$ and hence $f$ is a stationary point by the isometry of $A$. 
{On the other hand,  with noisy data  there is no regular solution to $|Ax|=b$ with high probability (since $A$ has many more rows than columns)  and  $f$ is unlikely to be a stationary point (since the stationarity equation imposes extra constraints on noise).}

Moreover,  the Hessian of \eqref{Gaussian} at $u=Af$ is positive semi-definite and has
 one-dimensional eigenspace spanned by $\im f$ associated with eigenvalue zero \cite{FDR,ptych-unique,AP-phasing}. 

{  Expanding the loss function \eqref{Gaussian} 
\beq\label{3'}
L(u)&=& \frac{1}{2}\|u\|^2-\sum_{j} b[j]|u[j]|+\frac{1}{2}\|b\|^2 
\eeq   
we see that $L$ has a {\em bounded} sub-differential  where $u[j]$ vanishes but $b[j]$ does not.
There are various tricks to further smooth out \eqref{Gaussian} e.g. by introducing an additional regularization parameter as
\beq
 L(u)&=&\half \| \sqrt{|u|^2+\ep}- \sqrt{b^2+\ep}\|^2,\quad\ep>0 \label{Gaussian'}
\eeq
(see e.g.  \cite{March19}). }

Besides the Poisson noise,
a type of noise due to interference from multiple scattering can be modeled as  complex circularly-symmetric Gaussian noise, resulting in the signal model
\beq
\label{ray}
b&=&|Af+\eta|
\eeq
where $\eta$ is a complex circularly-symmetric Gaussian noise. 
Squaring the expression, we obtain
\beq
b^2&=& |A f|^2+|\eta|^2+2\Re(\overline{\eta}\odot Af)\nn
\eeq

Suppose $|\eta|\ll |Af|$ so that $|\eta|^2 \ll 2\Re(\overline{\eta}\odot Af)$.
Then 
\beq
b^2&\approx &   |Af|^2+2\Re(\overline{\eta}\odot Af).\label{23-1}
\eeq
Eq. \eqref{23-1} says that at the photon counting level, the noise appears additive and Gaussian but with variance
proportional to $|Af|^2$, the Poisson noise in the asymptotic regime discussed in Appendix \ref{app:likelihood}. Therefore 
the loss function \eqref{Gaussian} is suitable for this case too.

The maximum likelihood scheme is a variance stabilization scheme which uniformizes the probability distribution for every pixel regardless of the measured intensity value \cite{noise3}. 
 See \cite{noise,noise2} for more choices of loss functions.

The amplitude-based Gaussian loss function
 \eqref{Gaussian} is well known to outperforms the
intensity-based loss function  $\half \| |u|^2- b^2\|^2$, even though the latter is more smooth \cite{Waller15}. Due to the non-differentiability of both $K$ and $ L$, the global convergence property
of the proposed DRS optimization is beyond the current framework of analysis \cite{LP}. 
The ptychographic iterative engines, PIE \cite{PIE104, PIE05, PIE204}, ePIE \cite{ePIE09} and rPIE \cite{rPIE17},  are also related to the mini-batch gradient method for the amplitude-based cost function \eqref{Gaussian}.

For $K$, 
we let $K(u)$ be the indicator function of the range of $A$, i.e. a ``hard" constraint.  

When the corresponding feasibility problem is consistent (feasible), there exist $u\in \IC^N$ such that
$|u|=b$ and $u=Ag$ for some $g\in \IC^{n^2}$, which are exactly the global minimizers of \eqref{DRS}, realizing the minimum value 0,  as well as the regular solutions defined in Section \ref{sec:solution}. 

When the corresponding feasibility problem is inconsistent (infeasible), the minimum value of \eqref{DRS} is unknown and  the global minimizers are harder to characterize.

DRS is based on the proximal operator which is a generalization of projection.
The proximal  point relative to a function $G$ is given by 
\[
\prox_{G}(u):= \mathop{\text{argmin}}\limits_{x} G(x)+\frac{1}{2}\| x-u\|^2.
\]
With the  loss functions \eqref{Poisson} or \eqref{Gaussian}, $P_X$ is replaced by 
$P_{K/\rho}$ and $P_Y$ by $P_{ L/\rho}$, respectively, with the step size $\gamma=1/\rho$. 
The 3-step procedure \eqref{dr2}-\eqref{dr2'} is replaced by 
\beq
\label{drs}
v_{l} &= &\prox_{K/\rho}(u_l); \\
w_{l} &= &\prox_{ L/\rho}(2v_{l}-u_l)\\
u_{l+1}&=& u_l+ w_{l}-v_{l}\label{drs'}
\eeq
 for $l=1,2,3\ldots$.

For convex optimization,  DRS \eqref{drs}-\eqref{drs'} is equivalent to
the Alternating Direction Method of Multipliers (ADMM) applied to the dual problem to \eqref{DRS}.
In Appendix \ref{sec:noisy-admm}, we show that for phase retrieval they are essentially equivalent
to each other.

For our choice of $K$, $ \prox_{K/\rho}(u)=P_X u=AA^+ u$ is independent of $\rho$. This
should be contrasted with the choice of the more smooth distance function adopted in \cite{LP} for the tractability of convergence analysis
(see more discussion in Section \ref{sec:last}). 

If we define the reflector $\cR_Y$ corresponding to $\prox_{ L/\rho}(u)$ as
\beq
\label{180}
\cR_Yu: =2\,\, \prox_{ L/\rho}(u)-u,
\eeq
then we can write the system \eqref{drs}-\eqref{drs'} as
\beq
\label{DRS1}
u_{k+1}=\half u_k+\half \cR_Y R_X u_k
\eeq
which is equivalent to 
\beq
\label{DRS2}
x_{k+1}=\half x_k+\half R_X \cR_Y x_k
\eeq
in terms of $x_k:=R_X u_k$. In other words, the order of carrying out 
$\prox_{ L/\rho}$ and $\prox_{K/\rho}$ does not matter in the current
DRS set-up.

For the Gaussian loss function \eqref{Gaussian}, the proximal point can be explicitly
derived 
\beqn
\prox_{L/\rho}(u) &=& \frac{1}{\rho+1}b\odot\sgn{(u)}+\frac{\rho}{\rho+1}u\\
&=& \frac{1}{\rho+1}(b+\rho|u|)\odot\sgn{(u)}, 
\eeqn
an averaged projection with the relaxation parameter $\rho$. Now we are ready to give
 the most compact and explicit representation of the Gaussian DRS map:
 \beq\label{G1}
u_{k+1} &=& \frac{u_k}{\rho+1} + \frac{\rho-1}{\rho+1}P_X u_{k}+\frac{1}{\rho+1}P_Y R_X u_k\\
&:=&\Gamma (u_k)\nn
\eeq 
which can be compared with AAR in the form \eqref{3step}.

For the Poisson case the DRS map has a more complicated form 
\beq	\label{P1}
		u_{k+1}& =&
		\half u_k-{1\over \rho+2} R_X u_k+
		      \frac{\rho}{2(\rho+2)}\lt[|R_Xu_{k}|^2+\frac{8(2+\rho)}{\rho^2}b^2\rt]^{1/2}\odot \sgn{\Big(R_X u_{k}\Big)}\\
		      &:=&\Pi(u_k)\nn
	\eeq
	where $b^2$ is the vector with component $b^2[j]=(b[j])^2$ for all $j$. 
	
Note that $\Gamma(u)$  and $\Pi(u)$ are continuous except where $R_X u$ vanishes but $b$ does not due to
arbitrariness of  the value of the $\sgn$ function at zero.

After the iteration is terminated with the terminal vector $u_*$, the object estimate is obtained by
\beq
f_*= A^+ u_*.\label{49'}
\eeq

We shall refer to DRS with the Poisson log-likelihood function \eqref{P1} and the Gaussian version \eqref{G1} by {\em Poisson-DRS} and {\em Gaussian-DRS}, respectively. The computation involved in Gaussian-DRS and Poisson-DRS are mostly pixel-wise operations
(hence efficient) except for the pseudo-inverse $A^+$ which can be computed efficiently (see Appendix \ref{app:matrix}). 

In the limiting case of $\rho= 0$, both Gaussian-DRS and Poisson-DRS  become the AAR algorithm.

\section{Fixed points}\label{sec:fixed}

For simplicity of presentation, we  shall focus on the case of the Gaussian DRS. 

 By definition, all  fixed points $u$  satisfy the equation 
\beq\label{sa}
u&=& \Gamma(u)
\eeq
and hence after some algebra by \eqref{G1} 
\beq
 \label{sa'} P_X u+\rho  P^\perp_X u=b\odot \sgn(R_X u). 
\eeq

The main result of this section is that the iteration of $\Gamma$ always produces a sequence bounded in norm
by 
\[
{\|b\|\over \min\{\rho, 1\}}\quad \mbox{\quad for  \quad} \quad \rho>0
\]
 (Theorem \ref{thm:bounded}) with  slightly better bounds on the fixed points (Corollary \ref{cor:norm}). Therefore, Gaussian-DRS is free of the divergence problem associated with AAR in the infeasible case.

It is often convenient to perform the  analysis in terms of the pair of variables $u$ and $x:=R_X u$. 
Here are some basic relations between $u$ and $x$. 
\begin{prop}
\label{prop3}
For any $u\in \IC^N$, $x:=R_X u$ satisfies 
\beqn
u =R_X x,& P_X u=P_X x, &P^\perp_X u=- P^\perp_Xx. 
\eeqn
\end{prop}
\begin{proof}
First note that 
\[
R_X x=2P_X x-x=2P_X u-(2P_X u-u)=u. 
\]
Moreover, 
\beqn
P_X x=P_XR_X u=P_X (2P_X u-u)=2P_X u-P_X u=P_X u.
\eeqn
and 
\beqn
P_X^\perp x=x-P_X x=2P_X u-u-P_X x= 2P_X u-u-P_X u=P_X u-u=-P_X^\perp u.
\eeqn
\end{proof}

\begin{prop}
\label{prop00}
Any $u\in \IC^N$ is a generalized solution if and only if  $x:=R_X u$ is a generalized solution.
\end{prop}
\begin{proof}
If $u$ is a generalized solution, then $P_X u=P_X x$ by Proposition \ref{prop3}. Now that
$x$ is a generalized solution, the converse is also true by the same argument. 

\end{proof}

\begin{prop}
\label{prop1} If $u$ is a generalized solution, 
then $P_X u$ is a regular solution and a fixed point. 
\end{prop}
\begin{proof} Let $\hat u= P_X u$. 
By Proposition \ref{prop1'} $\hat u$ is a regular solution.
Moreover $\Gamma(u)$ becomes
\beqn
 \frac{1}{2} \hat u+ \frac{\rho-1}{2(\rho+1)}\hat u +\frac{1}{\rho+1}b\odot \sgn (\hat u)
 &=&  \frac{1}{2} \hat u+ \frac{\rho-1}{2(\rho+1)}\hat u +\frac{1}{\rho+1} \hat u 
 \eeqn
 which equals $\hat u$. Therefore $\hat u$ is a fixed point. 
\end{proof}

\begin{prop}\label{prop2}
 Suppose $P_X u=u$.  Then 
$u$ is a regular solution if, and only if,  $u$ is a fixed point. 

\end{prop}
\begin{proof}
Under the assumption $P_X u=u$, $u=R_X u$ and $\Gamma(u)$ becomes
\beq
\label{381}
\half u+{\rho-1\over 2(\rho+1)} u +\frac{1}{\rho+1}b\odot \sgn \big(u)={\rho\over 1+\rho} u+\frac{1}{\rho+1}b\odot \sgn \big(u).
\eeq
Therefore, if $u$ is a fixed point,  then \eqref{sa'} implies 
\[
u=b\odot\sgn(u)
\]
and hence $|u|=b$, i.e. $u$ is a regular solution. 

On the other hand, if $|u|=b$, then the right hand side of \eqref{381} becomes
\[
{\rho\over 1+\rho} u+\frac{1}{\rho+1}b\odot \sgn \big(u)={\rho\over 1+\rho} u+\frac{1}{\rho+1}u
=u
\]
implying that $u$ is a fixed point.

\end{proof}

Writing 
\[
I=P_X+P^\perp_X\quad\mbox{and}\quad R_X=P_X-P^\perp_X, 
\]
and using Proposition \ref{prop3}
we can put the Gaussian-DRS map and the fixed point equation in the following forms.
\begin{prop}
\label{prop4}
The Gaussian-DRS map $\Gamma$  is equivalent to
\beq
\label{px}
\px u_{k+1}&=& {\rho\over \rho+1} \px u_k+{1\over \rho+1} \px\py x_k\\
\pxp u_{k+1}&=&{1\over \rho+1} \pxp u_k+{1\over \rho+1} \pxp\py x_k\label{pxp}
\eeq
where $x_k:=\rx u_k$. 
Therefore any fixed point $u$ satisfies
\beq
\label{px1}
\px x&=& \px\py x\\
-\rho \pxp x&=&  \pxp\py x,\label{pxp1}
\eeq
where $x:=\rx u$, or equivalently 
\beq
\label{32}
P_X x-\rho P^\perp_X x&=& b\odot\sgn(x)\\
\label{30}
P_X x+\rho P^\perp_X x & = & R_X \lt(b\odot \sgn{(x)}\rt).  
\eeq

\end{prop}

Next we show that  the Gaussian-DRS map $\Gamma$ with $\rho>0$ always produces a bounded iterated sequence, in contrast to the divergence behavior of AAR given in Proposition \ref{prop0} (ii). 
\begin{thm}\label{thm:bounded} \commentout{\red Suppose that $P_X$ is an orthogonal projection. }
Let $u_{k+1}:=\Gamma (u_k), \,\,k\in \IN,$ and $x_k:=R_X u_k$. Then,  for $\rho>0$, $\{u_k\}$ and $\{x_k\}$ are  bounded sequences. Moreover,
\beq\label{210}
\limsup_{k\to\infty} \|u_k\|= \limsup_{k\to\infty} \|x_k\|\le {\|b\|\over \min\{\rho, 1\}}&\mbox{\quad for  \quad}& \rho>0
\eeq
and hence  all fixed points $u$ satisfy
\beq
\label{217}
\|u\|\le {\|b\|\over \min\{\rho, 1\}}&\mbox{\quad for \quad}& \rho>0. 
\eeq
\end{thm}

\begin{proof}
Since $P_X$ is an orthogonal projection, we have 
\[
\|x_k\|=\|u_k\|=\sqrt{\|P_Xx_k\|^2+\|P_X^\perp x_k\|^2}. 
\]

By Proposition \ref{prop4} we then have the estimates
\beq
\label{212} \|u_{k+1}\|
&\le& \|{1\over \rho+1} P^\perp_X u_k+{\rho\over \rho+1}P_X u_k\|+{1\over \rho+1} \|P_Y x_k\|\\
&=& \lt[{1\over (\rho+1)^2}\|P^\perp_X u_k\|^2+{\rho^2\over (\rho+1)^2} \|P_X u_k\|^2\rt]^{1/2}+
{1\over \rho+1} \|P_Y x_k\|\nn\\
&\le& {\max\{\rho, 1\}\over \rho+1} \|u_k\|+{1\over \rho+1} \|b\|. \nn
\eeq

Hence, iterating \eqref{212} for $\rho\ge 1$ we obtain 
\beqn
\|u_{k+1}\|
&\le& {\rho^k\over (\rho+1)^k}\|u_1\|+{ \|b\|\over \rho+1}\sum_{j=0}^{k-1} {\rho^j\over (1+\rho)^j}
\eeqn
and,  after passing to the limit, the upper bound \eqref{210}.

On the other hand, for $\rho<1$,
\beqn
\|u_{k+1}\|&\le& {1\over (\rho+1)^k}\|u_1\|+\|b\|\sum_{j=1}^k {1\over (\rho+1)^j}
\eeqn
implying \eqref{210}. 
\end{proof}

We can improve  \eqref{217} slightly by Proposition \ref{prop4}. 
\begin{cor}
\label{cor:norm}
\commentout{\red Suppose $P_X$ is an orthogonal projection.}  For any fixed point $u$,  let $x:=R_X u$. Then 
\beq\label{270'}
\|u\|=\|x\| <\|b\|&\mbox{\quad if\quad }& \rho>1
\eeq
and
\beq\label{271}
\|b\|< \|u\|=\|x\|\le  \|b\|/\rho &\mbox{\quad if \quad}& \rho\in (0,1)
\eeq
unless $P_X x=x$  (or equivalently $P_X u=u$), in which case $u=x$ is a regular solution. 

On the other hand, for $\rho=1$, $\|u\|=\|x\|=\|b\|$ for any fixed point $u$.  
\end{cor}

\begin{proof}
By \eqref{32} and that $P_X$ is an orthogonal projection, we have
\beq
\label{321}
\|P_X x\|^2+\rho\|P^\perp_X x\|^2=\|b\|^2
\eeq
which implies 
\beq
\label{272}
\|u\|=\|x\|\lt\{\begin{matrix}
<\|b\|&\mbox{for}& \rho>1\\
>\|b\|&\mbox{for}& \rho<1
\end{matrix}\rt.\quad \mbox{ if \quad $\|P^\perp_X x\|\neq 0$. }
\eeq

If  $\|P^\perp_X x\|=0$, then $x=P_X x$ and \eqref{32} becomes $x=b\odot \sgn(x)$, implying $|x|=b$.  Likewise, $x=P_X x$
implies that $u=x$. 

Moreover, by \eqref{217}, $\|u\|=\|x\|\le \|b\|/\rho$ for $\rho\in (0,1)$. Hence \eqref{272} can be
further strengthened to the statement \eqref{270'}-\eqref{271}. 

For $\rho=1$, \eqref{321} implies that $\|x\|=\|b\|$. 
\end{proof}

In Appendix \ref{app:Poisson} we give a perturbation analysis 
for the similar result in the Poisson case with small $\rho$. 
\section{Stability analysis}\label{sec:stab}

When the uniqueness \eqref{unique} holds,  the fixed point set of AAR ($\rho=0$) is explicitly given in \eqref{fixed}. 
For $\rho>0$, the fixed point set is much harder to
characterize explicitly. Instead, we show that the desirable fixed points (i.e. regular solutions) are
automatically distinguished from the other non-solutional fixed points by their stability type. 

We say that a fixed point is
{\em   attracting}  if the spectral radius of the sub-differential map is at most 1 and {\em non-attracting} if otherwise.
Because a constant phase factor is an inherent ambiguity, any reasonable iterative map has at least one-dimensional {\em center manifold}.
We say that a fixed point is {\em strictly attracting} if the center manifold is one-dimensional, i.e. a positive spectral gap
between the second singular value of the sub-differential map and 1 (see Section \ref{sec:gap}).

Roughly speaking, we shall prove that
for $\rho\ge 1$ all   attracting fixed points must be regular solutions (Theorem \ref{thm:stable})
and that for $\rho\ge 0$ all regular solutions are   attracting (Theorem \ref{thm:stable2}). 
In other words, for $\rho\ge 1$, we need not concern with the problem of stagnation near a fixed point
that is not a regular solution (a common problem with AP). 
Moreover, we know that the regular solutions are strictly attracting under additional 
mild conditions (Corollary \ref{cor1}). On the other hand,
the problem of divergence (associated with AAR) when the data constraint is infeasible 
does not arise for Gaussian-DRS in view of Theorem  \ref{thm:bounded}.

\begin{prop} Let $
x: =R_X u$  and 
assume $|x|>0$. Set \beq
\label{380}
{\Omega}=\diag(\sgn(x)), \quad \tilde P_X={\Omega}^*P_X \Omega,  \quad \tilde R_X={\Omega}^*R_X \Omega. 
\eeq
 Then 
\beqn
\lim_{\ep \to 0}\Omega^*(\Gamma(u+\epsilon v) -\Gamma(u) )/\ep&= &J_A(\eta),\quad \eta ={\Omega}^*v
\eeqn
where 
\beq
\label{38'}
J_A(\eta) 
&=& \half \eta+ {\rho-1\over 2(\rho+1)} \tilde R_X \eta+{\im\over \rho+1} {b\over |x|} \odot \Im\lt[\tilde  R_X \eta \rt].
\eeq

\end{prop}

\begin{proof}
 The key observation is 
that the derivative of $\sgn(c)=c/|c|\in \IC, c\neq 0$,  is given by 
\beqn
\lim_{\ep\to 0} {1\over \ep } \lt[{c+\ep  a\over |c+\ep a|} -{c\over |c|}\rt]&=&\lim_{\ep\to 0} {\sgn(c)\over \ep } \lt[{1+\ep  a/c\over |1+\ep a/c|} -1\rt]\\
&=&\im\, \Im\lt[{a/c}\rt]{\sgn(c)}\\
&=&\im\, \Im\lt[\sgn(\bar c) {a}\rt]{\sgn(c)\over |c|}
\eeqn
{for any $a \in \IC$} where $\Im$ denotes the imaginary part. 
So we have
\beqn
\lim_{\ep \to 0}{1\over \ep}(\Gamma(u+\epsilon v) -\Gamma(u) )&= &
\half v+ {\rho-1\over 2(\rho+1)} R_X v+{\im\over \rho+1} {b\over |x|} \odot \Omega \Im\lt[\Omega^* R_X v\rt]
\eeqn
which,  in terms of $\eta=\Omega^* v$ and the notation \eqref{380}, becomes $\Omega$ times  $J_A$ in \eqref{38'}.

\end{proof}

The following result says that for $\rho\ge 1$ all the non-solution fixed points are non-attracting.
 \begin{thm} Let $\rho\ge 1$. \commentout{\red and suppose that $P_X$ is an orthogonal projection.}
Let  $u$ be a fixed point such that  $x: =R_X u$ has no vanishing components. Suppose 
 \beq
 \label{270}
\|J_A(\eta)\| \le \|\eta\|,\quad\forall \eta\in \IC^N. 
   \eeq
Then
\beq
\label{key2}
 x=P_X x=b\odot\sgn(x),
\eeq
implying $u=x$ is a regular solution. 
\label{thm:stable}
\end{thm}
\begin{rmk}\label{rmk:cdr}
Previous results \cite{FDR} suggest that when the regular solution is unique
up to a constant factor, all AAR fixed points in \eqref{fixed2'} are   attracting in the sense \eqref{270}. In other words, Theorem \ref{thm:stable} is likely false  for $\rho=0$. 
\end{rmk}

\begin{proof}
In view of Proposition \ref{prop4}, it suffices to show that $P^\perp_X x=0$. 

We prove the statement by contradiction.  Suppose $P^\perp_X x\neq 0$. 

By  \eqref{32} and the Pythogoras theorem 
we have 
\beq
\|P_X x\|^2_2+\rho^2 \|P^\perp_X x\|^2=\|b\|^2\label{419}
\eeq
and hence $\|b\|\ge \|x\|$ for $\rho\ge 1$.
Applying $\Omega^*$  we rewrite \eqref{30} as 
\beq\label{30'}
\tilde P_X |x|+ \rho(|x|-\tilde P_X |x|)&=& \tilde R_X b
\eeq
On the other hand, applying $\tilde P_X$ on \eqref{30'} we have 
\begin{equation*}
\tilde P_X |x| = \tilde P_X b
\end{equation*}
and hence by \eqref{30'} 
\beq
\label{38}
\tilde P_X |x| = \tilde P_X b={\rho |x|\over 1+\rho}+{b\over 1+\rho}.\label{key}
\eeq

We now show that $\|J_A(\eta)\| > \|\eta\|$ for any $\eta$ such that 
\beq
\label{414}
\tilde R_X\eta= \im \tilde P_X b={\im\rho \over 1+\rho}|x|+{\im \over 1+\rho} b. 
\eeq
To this end, it is more convenient to write $J_A$ in \eqref{38'} in terms $\xi:= \tilde R_X \eta$. With a slight abuse of notation
we write
\beq
\label{383}
J_A(\xi)
&=&\tilde P_X \xi-{\xi\over \rho+1}+ {\im\over \rho+1}{b\over |x|} \odot \Im(\xi)
\eeq
where we have used the properties in Proposition \ref{prop3}.

Since $\|\xi\|=\|\eta\|$, our goal is to show 
$\|J_A(\xi)\|>\|\xi\|$. 

First we make an observation that will be useful later. We claim that 
\beq
\label{415}
\rho\|x\|^2=\|b\|^2+(\rho-1)|x|\cdot b
\eeq
where ``$\cdot$" denote the (real) scalar product between two vectors. 
By \eqref{key}, 
\beqn
\tilde P^\perp_X b&= b- \tilde P_X b = {\rho \over 1+\rho}(|b|-|x|)
\eeqn
and hence by the Pythogoras theorem
\beqn
\|b\|^2&=& \|\tilde P_Xb\|^2+\|\tilde P^\perp_Xb\|^2\\
&=& \lt\|{\rho |x|\over 1+\rho}+{b\over 1+\rho}\rt\|^2+\lt\|{\rho \over 1+\rho}(|b|-|x|)\rt\|^2\\
&=&{2\rho^2\over(\rho+1)^2}\|x\|^2+{2\rho(1-\rho)\over (\rho+1)^2} |x|\cdot b+
{\rho^2+1\over (\rho+1)^2} \|b\|^2
\eeqn
which becomes \eqref{415} after rearrangement. 

Next, note that by \eqref{414} 
\[
\tilde P_X \xi=\tilde P_X \eta=\im \tilde P_X b=\xi, 
\]
which is purely imaginary,  and hence
\beq
\label{421}
J_A(\xi)&=&{\rho\over \rho+1}\xi + {\im\over \rho+1}{b\over |x|} \odot \xi
\eeq
by \eqref{383}.

After some tedious but straightforward algebra with \eqref{414} and \eqref{421}, we see that $\|J_A(\xi)\|>\|\xi\|$ 
is equivalent to the inequality
\beqn
0<(5\rho^2-2\rho-1)\|b\|^2+(2\rho^3-4\rho^2-2\rho) |x|\cdot b+4\rho {b\over |x|}\cdot b^2+
\lt\|{b^2\over |x|}\rt\|^2-\rho^2(2\rho+1)\|x\|^2
\eeqn 
which by \eqref{415} reduces to 
\beq
\label{441}
0< (3\rho^2-3\rho-1)\|b\|^2-(3\rho^2+\rho)|x|\cdot b+
4\rho {b\over |x|}\cdot b^2+\lt\|{b^2\over |x|}\rt\|^2.
\eeq

To proceed, we note that the assumption $\tilde P_X^\perp x\neq 0$ implies 
$|x| \neq b$, $\|x\|<\|b\|$ and moreover $|x|, b$ are not a multiple of each other. 
So by the Cauchy-Schwarz inequality we have 
\beqn
\lt\|{b^2\over |x|}\rt\| &>&{\|b\|^2\over \|x\|}\\
{b\over |x|}\cdot b^2 &=&\lt\|{b^{3/2}\over |x|^{1/2}}\rt\|^2> {\|b\|^4\over \| |x|^{1/2}\odot b^{1/2}\|^2}= {\|b\|^4\over |x|\cdot b}. 
\eeqn
and hence the last two terms on the right hand side of \eqref{441} have the strict lower bound
\beq
\label{442}
4\rho {b\over |x|}\cdot b^2+\lt\|{b^2\over |x|}\rt\|^2&> & 4\rho {\|b\|^4\over |x|\cdot b}+ {\|b\|^4\over \|x\|^2} \\
&> &(1+4\rho)\|b\|^2\nn
\eeq
where we have used the fact $\|b\|\ge\|x\|$ due to $\rho\ge 1$.

In view of \eqref{442}  the right hand side of \eqref{441} is strictly greater than
\[
(3\rho^2-3\rho-1)\|b\|^2-(3\rho^2+\rho)\|b\|^2+(1+4\rho)\|b\|^2=0.
\]
In other words, \eqref{441} holds indeed and the proof for $\|J_A(\xi)\|>\|\xi\|$ is complete.

 This clearly  contradicts the assumption  \eqref{270}. 
  Therefore, $P^\perp_X x=0$ and the desired result \eqref{key2}  follows from Propositions \ref{prop2} and \ref{prop4}. 
 
\end{proof}

 The next result says that for any $\rho\ge 0$,  all regular solutions are
 attracting fixed points.
 
 \begin{thm}\label{thm:stable2} Let $\rho\ge 0.$ \commentout{\red and suppose $\px$ is an orthogonal projection. }
Let $u$ be  a nonvanishing regular solution. 
Then 
\beq
\label{320}
\|J_A(\eta)\| \le \|\eta\|
\eeq
 for all $\eta \in \IC^N$ and the equality holds 
in the direction $\pm \im b$ (and possibly elsewhere on the unit sphere). 
\end{thm}
\begin{proof}
By Proposition \ref{prop2}, $x:=R_X u=u$ is a fixed point. By Proposition \ref{prop4}, 
\[
u=b\odot \sgn(u)=Ag \quad\mbox{for some $g$.}
\]
  Rewriting $J_A(\eta)$ in \eqref{38'} as 
\beqn
J_A(\eta) &= & \tilde P_X \eta-{1\over \rho+1} \tilde R_X \eta+{\im\over 1+\rho}{b\over |x|}\odot \Im\lt(\tilde R_X \eta\rt)
\eeqn
and using  $|x|=b$ we obtain
\beqn
 J_A (\eta)&=&\tilde P_X\eta- {1\over 1+\rho}\Re\big(\tilde R_X \eta \big)\nn\\
\eeqn
where $\Re$ denotes the real part. 
We now show that $\|J_A(\eta)\|\le \|\eta\|$ for all $\eta$. 

To proceed, we shall write $\tilde P_X=HH^*$ where $H$ is an isometry. 
This can  be done for the matrix $C:=\Omega^* A$ via the QR decomposition. In our setting, the measurement matrix of each diffraction pattern has orthogonal columns and so does  the total measurement matrix. Hence the $R$ factor of $C$ is
a diagonal matrix with the norms of the columns of $C$ on the diagonal (see Appendix \ref{app:matrix}). 
For ease of notation, we may assume that $\Om^* A =H$. 

Note that 
\beq\label{Bv'}
\lt[\begin{matrix}
\Re[H]&- \Im[H]\\
\Im[H]& \Re[H]
\end{matrix}\rt]
\eeq
is real isometric because $H$ is complex isometric.
Define  \begin{equation} \label{Bv}
\cH:=\left[
\begin{matrix}
 \Re[H]&
  \Im[H] 
\end{matrix}
\right] \in \IR^{N\times 2n^2}.
\end{equation}

As in the set-up detailed in Appendix \ref{app:matrix} let the object be a $n\times n$ square image and
$\IC^{n^2}$ the object domain.

By Proposition \ref{Srate} in Appendix \ref{app:eigen}, $HH^*$ 
can be block-diagonalized into one $(N-2n^2)\times (N-2n^2)$ zero-block and $2n^2\,\,$ $2\times 2$ blocks of the form
\beq\label{block}
\lt[ \begin{matrix}
\lambda_k^2& \lambda_k\lambda_{2n^2+1-k}\\
\lambda_k\lambda_{2n^2+1-k}&\lambda_{2n^2+1-k}^2
\end{matrix}\rt],\quad k=1,2,\ldots, 2n^2
\eeq
in the orthonormal basis $\{\eta_k, \im \eta_{2n^2+1-k}: k=1,2,\ldots, 2n^2\}$
where $ \eta_k\in \IR^N$ are the right singular vectors, corresponding to the singular values $\lamb_k$,  of $\cH$.

Moreover, the complete set of singular values satisfy  
\beq
\label{40}
&1=\lamb_1\ge \lamb_2\ge \ldots\ge \lamb_{2n^2}= \lamb_{2n^2+1}=\ldots=\lamb_{N}=0&\\
&\lamb_k^2+\lamb_{2n^2+1-k}^2=1.&\label{41}
\eeq

In view of the block-diagonal nature of $HH^*$, we shall analyze $J_A(\eta)$ in the 2-dim spaces spanned by the orthonormal basis $\{\eta_k, \im \eta_{2n^2+1-k}\}$ one $k$ at a time. 

For any fixed $k$ and any $ z_1, z_2\in \IC$  let 
 \beqn
 \eta&=&z_1 \eta_k+\im z_2\eta_{2n^2+1-k}\\
 &=& \Re[z_1]\eta_k+ \Re[z_2] \im \eta_{2n^2+1-k}+\Im[z_1]\im \eta_k
 -\Im[z_2]\eta_{2n^2+1-k}.\nn
 \eeqn
 We shall use the basis $\{\eta_k,\im\eta_{2n^2+1-k}, \im \eta_k, -\eta_{2n^2+1-k}\}$  for expressing $\eta$ and $J_A(\eta)$.
 
 We obtain 
\beq\label{50}
J_A(\eta) &=&\lt(\lamb_k^2z_1+\lamb_k\lamb_{2n^2+1-k}z_2\rt)\eta_k+
\lt(\lamb_k\lamb_{2n^2+1-k}z_1+\lamb^2_{2n^2+1-k}z_2\rt)\im\eta_{2n^2+1-k}\\
&&+{1\over 1+\rho}\lt[
(1-2\lamb_k^2)\Re(z_1)-2\lamb_k\lamb_{2n^2+1-k} \Re(z_2)\rt]\eta_k \nn\\
&&
+ {1\over 1+\rho}\lt[2\lamb_k\lamb_{2n^2+1-k} \Im(z_1)-(1-2\lamb_{2n^2+1-k}^2)\Im(z_2) \rt]\eta_{2n^2+1-k}. \nn
\eeq
Next we treat \eqref{50} as a linear function of $\Re(z_1), \Re(z_2), \Im(z_1), \Im(z_2)$ with real coefficients in the basis $\{\eta_k,\im\eta_{2n^2+1-k}, \im \eta_k,- \eta_{2n^2+1-k}\}$ and represent $J_A$ by
a $4\times 4$ matrix which is block-diagonalized into two $2\times 2$ blocks:
\beq\label{51}
\lt[\begin{matrix} {1\over 1+\rho}+{\rho-1\over\rho+1} \lamb_k^2 & {\rho-1\over \rho+1}\lamb_k\lamb_{2n^2+1-k}\\
\lamb_k\lamb_{2n^2+1-k}& \lamb_{2n^2+1-k}^2
\end{matrix}\rt],\quad
\lt[\begin{matrix}
\lamb_k^2 &\lamb_k\lamb_{2n^2+1-k}\\
{\rho-1\over \rho+1}\lamb_k\lamb_{2n^2+1-k} &{1\over \rho+1}+{\rho-1\over\rho+1} \lamb_{2n^2+1-k}^2
\end{matrix}\rt].
\eeq
with the former of \eqref{51} acting on $\Re(z_1), \Re(z_2)$ and 
the latter acting on $\Im(z_1), \Im(z_2)$. 

The eigenvalues  of  the matrices in \eqref{51} are, respectively
\beq
\label{52}
&&{1\over 2(\rho+1)} \lt[{\rho} +{2} \lamb_{2n^2+1-k}^2\pm \sqrt{\rho^2-4\lamb_k^2\lamb^2_{2n^2+1-k}}\rt]\\
&&{1\over 2(\rho+1)} \lt[{\rho} +{2} \lamb_k^2\pm \sqrt{\rho^2-4\lamb_k^2\lamb^2_{2n^2+1-k}}\rt].\label{522}
\eeq
Because the two expressions are symmetrical with respect to the exchange of index ($k\leftrightarrow 2n^2+1-k$),
it suffices to analyze \eqref{52}, 
which, with the + sign,  equals 1 at $k=2n^2$ (recall $\lamb_{2n^2}=0$).
Next we  show that 1 is the largest eigenvalue
 among all $k$ and $\rho\in [0,\infty)$.

Note that  \eqref{52} is real-valued for any $\lamb_k\in [0,1]$ if and only if $\rho\ge 1$.
Hence, for $\rho\ge 1$, the maximum eigenvalue is 1 and occurs at $k=2n^2$.

For $\rho<1$ and $\rho^2-4\lamb_k^2+4\lamb_k^4\ge 0$,  the maximum value of \eqref{52} 
can be bounded as
\beq
\label{52'}
&& {1\over 2(\rho+1)} \lt[{\rho} +{2} (1-\lamb_k^2)+\sqrt{\rho^2-4\lamb_k^2(1-\lamb_k^2)}\rt]\\
&\le& {1\over 2(\rho+1)} \lt[{\rho} +{2} (1-\lamb_k^2) +\rho\rt]\nn\\
&=& 1-{\lamb_k^2\over \rho+1}\le 1\nn
\eeq
since $4\lamb_k^2(1-\lamb_k^2)\ge 0$. 
Hence the expression \eqref{52} achieves
the maximum value 1 at $k=2n^2$.

For $\rho<1$ and $\rho^2-4\lamb_k^2+4\lamb_k^4\le 0$,
the modulus of \eqref{52} equals
\beq
\label{55}
\sqrt{1-\lamb_k^2\over 1+\rho}\le 1. 
\eeq

By Proposition \ref{B*bound}, $\eta_1=b,$ $\tilde P_X (\im b)=\im b,$ $\tilde R_X (\im b)=\im b$ and hence $J_A(\im b)=\im b$. However, $\arg\max_{\|\eta\|=1}|J_A(\eta)|$ may contain points other than
$\pm\im b/\|b\|$ 
since we do not know if $\lamb_2<1$ without additional conditions (see Section \ref{sec:gap}). The proof is complete.

\end{proof}

\section{Spectral gap}\label{sec:gap}

To derive  a positive spectral gap ($\lamb_2<\lamb_1=1$), we need some details of the ptychographic set-up (Appendix \ref{app:matrix}).   

Let $\cT$ be the set of all shifts,  including $(0,0)$,  involved  in the ptychographic measurement. 
 Denote by $\mu^\bt$ the $\bt$-shifted probe for all $\bt\in \cT$ and $\cM^\bt$ the domain of
$\mu^\bt$. Let $f^\bt$ the object restricted to $\cM^\bt$.  For convenience, we assume  the periodic boundary condition on the whole object domain $\cM=\cup_{\bt \in \cT} \cM^\bt$ when $\mu^\bt$ crosses over the boundary of $\cM$. 

 \begin{figure}[t]
\begin{center}
\subfigure[]{\includegraphics[width=5cm]{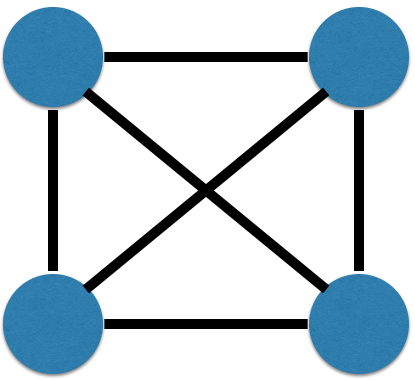}}\hspace{2cm}
\subfigure[]{\includegraphics[width=5cm]{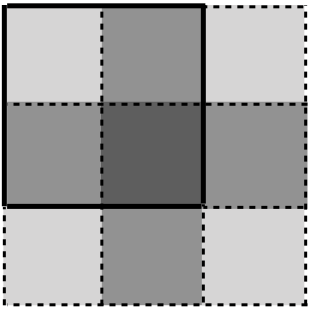}}
\caption{A complete undirected graph (a) representing four connected object parts (b) where the gray level indicates the number of coverages by the mask in four scan positions.
}
\label{fig:graph}
\end{center}
\end{figure}

Two blocks $\cM^\bt$ and $\cM^{\bt'}$ are said to be connected if  the minimum overlap condition 
\[
\#\{\cM^{\bt}\cap \cM^{\bt'}\cap \supp(f)\}\ge 2
\]
 is satisfied. Let $\cG$ be the undirected graph with the nodes corresponding to $\{\cM^\bt: \bt\in \cT\}$ and the edges between any pair of connected nodes (see Figure \ref{fig:graph}).

 Now we recall the following spectral gap theorem \cite{ptych-unique} (Proposition 3.5 and the subsequent remark).
   \begin{prop}\cite{ptych-unique}\label{prop:gap}  In addition to the above assumptions, 
 suppose  $\supp(f)$ is not a subset of a line. Let $u$ (and hence $x:=R_X u$) be a regular solution. 
 Let $\lamb_2$ be the second largest singular value of $\cH$ defined in \eqref{Bv}.  If the graph $\cG$ is connected,  then $\lamb_2<1$.
 \end{prop}
 Some theoretical bounds for $\lambda_2$ can be found in \cite{ptych-unique}. 

Using Proposition \ref{prop:gap}, we can sharpen the result of Theorem \ref{thm:stable2}
as follows.  
  
   \begin{cor}\label{cor1} Under the assumptions of Proposition \ref{prop:gap}, the second largest singular value of $J_A$ is strictly less than 1 and achieves
 the minimum value
 \beq
 \label{450}
 {\lamb_2\over \sqrt{1+\rho_* }}\quad  \mbox{at}\quad \rho_*=2\lamb_2\sqrt{1-\lamb_2^2}\in [0,1].
 \eeq
 Moreover, the second largest singular value of $J_A$ is an increasing function of $\rho$ in
 the range $[\rho_*,\infty)$ and a decreasing function   in the range of $[0,\rho_*]$. 
 
 \end{cor}
 \begin{rmk} \label{rmk6.3} 
 By arithmetic-geometric-mean inequality,
 \[
 \rho_*\le 2 \times \half \sqrt{\lamb_2^2+1-\lamb_2^2}=1
 \]
 where the equality holds only when $\lamb^2_2=1/{2}$. 
 
 As $\lamb^2_2$ tends to 1, $\rho_*$ tends to 0 and as $\lamb^2_2$ tends to $\half$, $\rho_*$ tends to 1.
 Recall that $\lamb_2^2+\lamb^2_{2n^2-1}=1$ and hence $ [1/2,1]$ is the proper range of $\lamb^2_2$. 

 \end{rmk}
 
 \begin{proof}
 Our discussion splits into several cases. By the identity $\lamb_2^2=1-\lamb_{2n^2-1}^2$,  we have $\lamb_2^2(1-\lamb_2^2)=\lamb_{2n^2-1}^2(1-\lamb_{2n^2-1}^2)$ and 
 $\lamb_2^2\ge 1/2$. 
 
For $\rho>1$,  the larger eigenvalue in \eqref{52}  achieves the second largest value
 \beq
 \label{460}
 {1\over 2(1+\rho)} \lt[\rho+2\lamb_2^2+\sqrt{\rho^2-4\lamb_2^2(1-\lamb_2^2)}\rt]
\eeq
 at $k=2n^2-1$ after some algebra. The expression \eqref{460} is strictly less than
 \beqn
  {1\over 2(1+\rho)} \lt[2\rho+2\lamb_2^2\rt]={\rho+\lamb_2^2\over \rho+1}< 1 
  \eeqn
 with the spectral gap $\lamb_2<1$.

For $\rho=1$, \eqref{52} becomes
\beqn
{1\over 4} \lt[1+2(1-\lamb_k^2)\pm |1-2\lamb_k^2|\rt] 
\eeqn
which achieves the second largest value
\beq
\label{54'}
1-\lamb_{2n^2-1}^2=\lamb_2^2<1
\eeq
at $k=2n^2-1$ by \eqref{41}. 

The case of $\rho<1$ requires more analysis since the eigenvalue \eqref{52} may be real or complex. Analyzing as in  \eqref{52'} and \eqref{55}
we  conclude
that  the second largest value is 
\beq
\label{481}
{1\over 2(\rho+1)} \lt[{\rho} +{2}\lamb_2^2+\sqrt{\rho^2-4\lamb_2^2(1-\lamb_2^2)}\rt] &\mbox{\quad if \quad}& \rho\ge \rho_*=2\lamb_2\sqrt{1-\lamb_2^2}
\eeq
and
\beq
\label{482}
 {\lamb_2\over \sqrt{1+\rho}}&\mbox{\quad if \quad}& \rho\le \rho_*
\eeq
While the expression in \eqref{482} is a decreasing function of $\rho$ and less than
$\lamb_2$, \eqref{481} 
is an increasing function of $\rho$ and less than \eqref{54'} for $\rho=1$.

Also, for $\rho>1$, the expression \eqref{460}, as a function of $\rho$, has the derivative
\[
{1\over 2(\rho+1)^2} \lt[ 1-2\lamb_2^2+ {\rho+4\lamb_2^2(1-\lamb_2^2)\over \sqrt{\rho^2-4\lamb_2^2(1-\lamb_2^2)}}\rt]>{1\over 2(\rho+1)^2} \lt[ 2-2\lamb_2^2\rt]>0
\]
and hence achieves the minimum at $\rho=1$.
In other words, Gaussian-DRS with $\rho=1$ converges faster than Gaussian-DRS with $\rho>1$.

From the preceding analysis, the second largest singular value achieves the minimum
at the crossover value $\rho_*$ of the two expression in \eqref{481}. 
Substituting $\rho_*$ in \eqref{481} we arrive at \eqref{450}.

Although it is not immediately obvious, it can be verified by elementary (but somewhat tedious) calculus that \eqref{450} is less than $\lamb_2^2$ (for $\rho=1$).

 \end{proof}
 
 For comparison with AAR, we note 
 that, for $\rho=0$,  \eqref{482} is exactly $\lambda_2$ and hence greater than $\lamb_2^2$ in  \eqref{54'}, the convergence rate for $\rho=1$,  
 which coincides with the convergence rate of Alternating Projections (AP) \cite{AP-phasing}. We state this observation as a corollary. 
 
 \begin{cor}\label{cor:rate1}
 For the Gaussian-DRS with $\rho=1$, the local convergence rate is given by $\lamb_2^2$ which is smaller
 than the convergence rate $\lamb_2$ for $\rho=0$. 
 
 \end{cor}
  
{  With a positive spectral gap, this largest eigenvalue $1$ in Theorem \ref{thm:stable2} corresponds to the global phase factor and Corollary \ref{cor1}
can be used to prove local, linear convergence for Gaussian-DRS with $\rho\ge 0$. 
The proof is analogous to that in \cite{ptych-unique}  for phase retrieval (Theorem 5.1)  and in \cite{FDR} (Theorem 3.4)  for ptychography for AAR ($\rho=0$). But the argument is technical in nature and omitted here  for the sake of space.}

\section{Selection of parameter}\label{sec:stepsize}
{  A goal of the present work is to circumvent the divergence behavior of AAR (as stated in Proposition \ref{prop0} (ii) for the convex case)
when the feasibility problem is inconsistent and has no (generalized or regular) solution.  

Let us first examine how this problem manifests in the fixed point equation \eqref{sa'} reproduced here for the convenience of the reader:
\beqn
P_X u+\rho  P^\perp_X u=b\odot \sgn(R_X u). 
\eeqn
For $\rho=0$, $P_X u=b\odot \sgn(R_X u) $ and, in particular, $|P_X u|=b$, i.e. every fixed point of AAR  is a generalized solution. 
So if the problem is inconsistent, then no solution (generalized or regular) exists, implying the fixed point set is empty. 

The case with $\rho>0$ is harder to analyze. For $\rho\ge 1$, however, Theorem \ref{thm:stable} says that all   attracting fixed points are regular solutions and hence   in the inconsistent case  all fixed points are
 repelling in some directions (likely partially hyperbolic with a center manifold containing at least a circle corresponding to an arbitrary constant phase factor). In other words convergence to a fixed point is almost impossible  in the inconsistent case with $\rho\ge 1$. From this perspective, Theorem \ref{thm:stable} is a pessimistic result in the traditional sense of convergence analysis. 

But all hope is not lost.  First of all, let us recall the earlier observation that in the inconsistent case $f$ is probably not a stationary point of the loss function. Hence a convergent iterative scheme to a stationary point may not be a good idea after all. A good iterative scheme need not converge as long as it produces a good outcome when properly terminated, i.e. its iterates stay in the true solution's vicinity of size comparable to the noise level. 

 Second, Theorem \ref{thm:bounded} implies  that every
Gaussian-DRS sequence is bounded and has a convergent subsequence $\{u_{k_j}\}_{j=1}^\infty$ with the limit, say $\hat u$. 
If, in addition, 
\beq
\label{strange}
\lim_{j\to\infty} (u_{k_j}-\Gamma(u_{k_j}))=0,
\eeq
 then by taking the limit on both sides of the fixed point equation \eqref{sa'}, one can conclude that
$\hat u$ is a fixed point. 
The preceding analysis tells us that  in the inconsistent case  \eqref{strange} is false for $\rho\ge 1$ (The case with $\rho\in (0,1)$ is open), suggesting that $\hat u$ is part of a more complicated attractor. 

In particular, if $\hat x:=R_X \hat u$ does not vanish where $b>0$, then, by the continuity of $\Gamma$ at such points,   $\hat u_1:=\lim_j \Gamma(u_{k_j})$ exists.
Assuming that $R_X\Gamma^l(u_{k_j}), l\ge 1$, do not vanish wherever $b>0$, we obtain a set of cluster points $\hat u_l=\lim_j \Gamma^l(u_{k_j}), l\ge 1$
which constitutes a new iterated  sequence, i.e. $\hat u_{l+1}=\Gamma(\hat u_l)$. This is the case of limit cycle in theory of bifurcation. 
If, however, the non-vanishing assumption  fails, then different orbits can branch off at discontinuities. 

In general, when a bounded invariant set exists (as implied by Theorem \ref{thm:bounded}) and no fixed point is attracting (e.g., with $\rho\ge 1$ in the inconsistent case), there tend to be some nontrivial attractors (limit cycles,  strange attractors, ergodic invariant domain etc).

Nevertheless, the non-convergent sequence controlled by the underlying attractors may still produce  a reasonable solution under a proper termination rule. Our numerical experiments with noisy data confirms that  this is indeed the case (see Figure \ref{fig:noise}). Analyzing the properties of such attractors is at the frontier of numerical analysis and beyond the scope of the present work. 
}

{  
\subsection{Phase retrieval with noiseless data}
\label{sec:num2}
\begin{figure}[t]
\centering
\subfigure[]{\includegraphics[width=6.5cm]{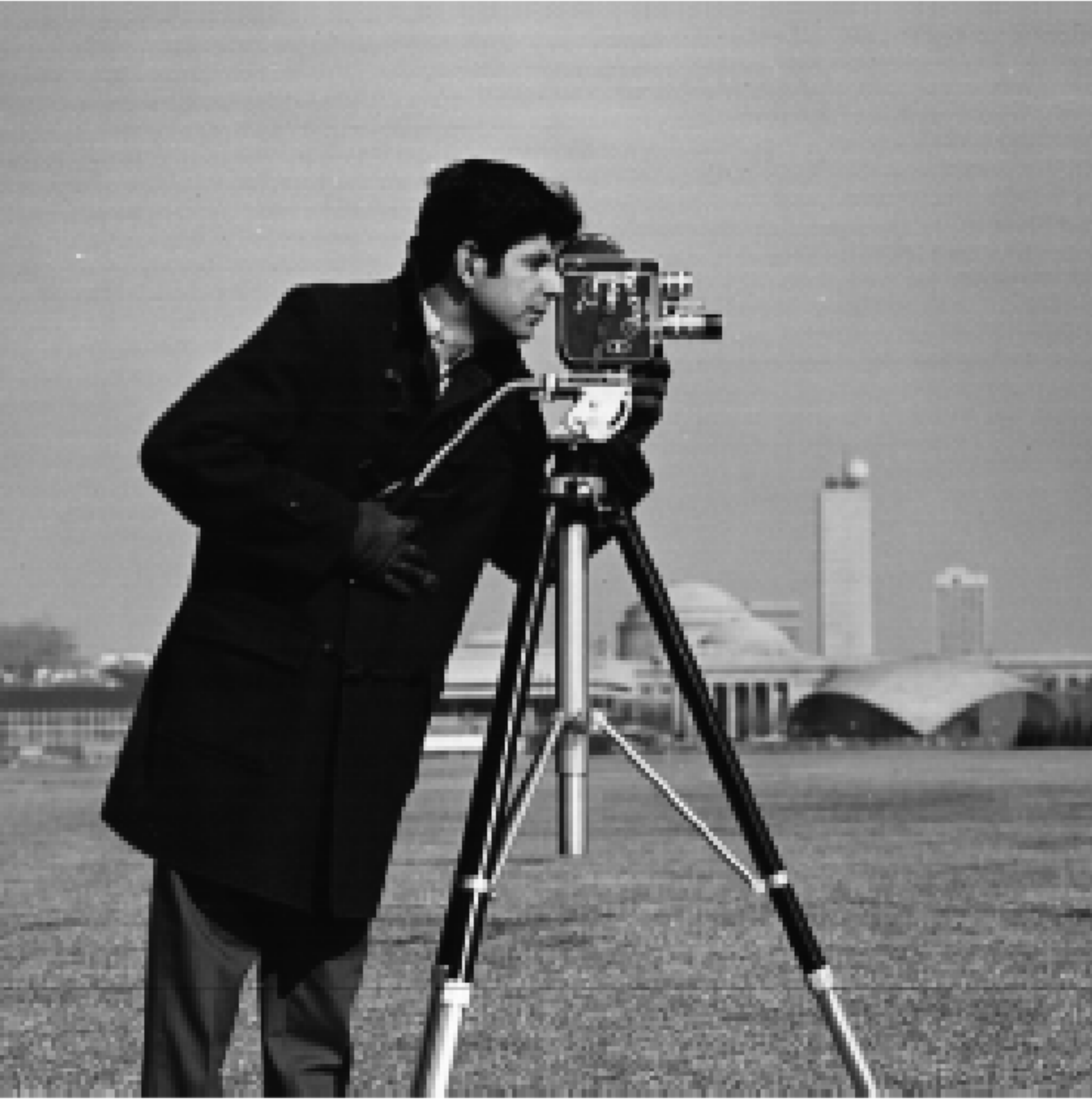}}\hspace{1cm}
\subfigure[]{\includegraphics[width=6.5cm]{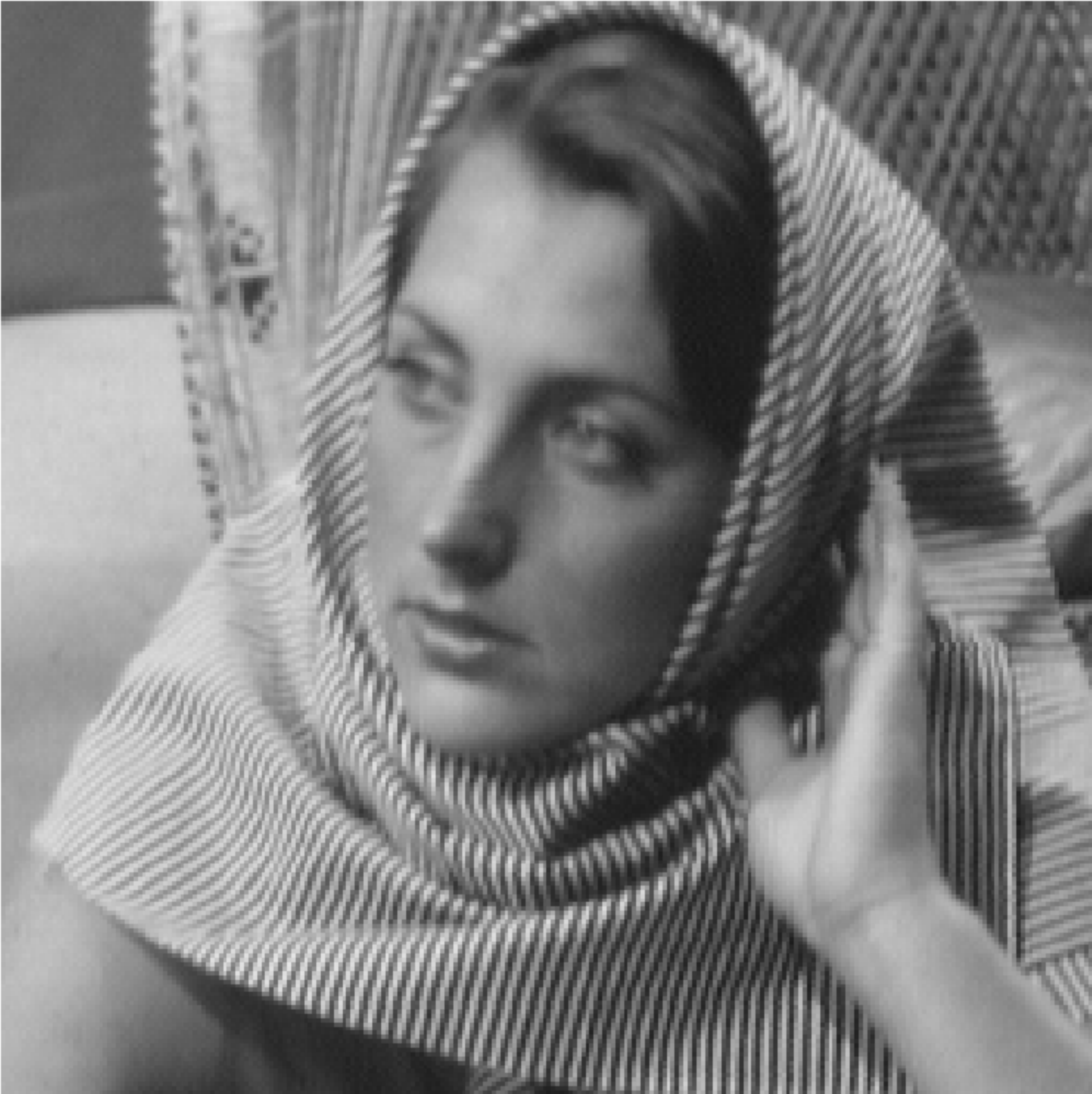}}
  \caption{(a) The real part and (b) the imaginary part of the test image CiB.}
  \label{fig:image}
\end{figure}%

We conduct a brief exploration of the optimal parameter for Gaussian-DRS. 
Our test image  is 256-by-256 Cameraman+ $\im$ Barbara (CiB). The resulting test object has the phase range $\pi/2$. 

We use three baseline algorithms as benchmark. The first is AAR. 

The second is Gaussian-DRS with $\rho=1$ 
\beq
\Gamma_1(u)&=&\half u+\half P_YR_X u\label{Gdrs}
\eeq
(since $\cR_Y$ in \eqref{180} is exactly $P_Y$ with $\rho=1$)
given the basic guarantee that for $\rho\ge 0$ the regular solutions are   attracting (Theorem \ref{thm:stable2}), that  for the range $\rho\ge 1$ no fixed points other than the regular solution(s) are locally attracting (Theorem \ref{thm:stable}) and that 
Gaussian-DRS with $\rho=1$ produces the best convergence rate for any $\rho \ge 1$ (Corollary \eqref{cor1}).
 The contrast between \eqref{Gdrs}  and AAR \eqref{dr} is noteworthy. 
The simplicity of the form \eqref{Gdrs} suggests the name {\em Averaged Projection Reflection} (APR) algorithm.

\begin{figure}
\centering
\subfigure[$\rho=1.1,\beta=0.9$]{\includegraphics[width=7cm]{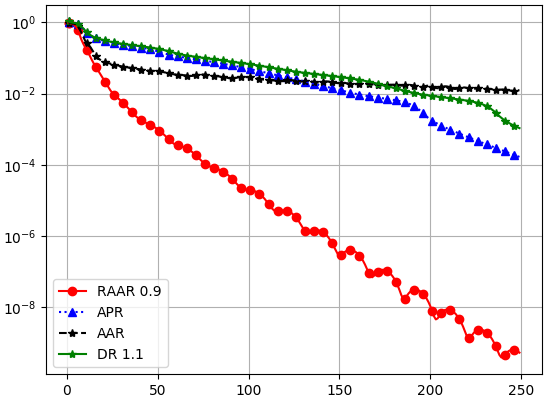}}
\subfigure[$\rho=0.5,\beta=0.9$]{\includegraphics[width=7cm]{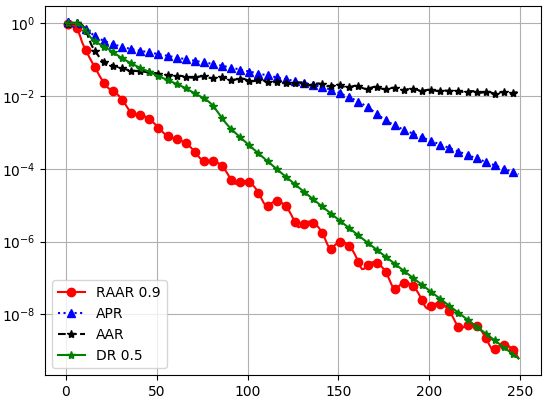}}
\subfigure[$\rho=0.3,\beta=0.9$]{\includegraphics[width=7cm]{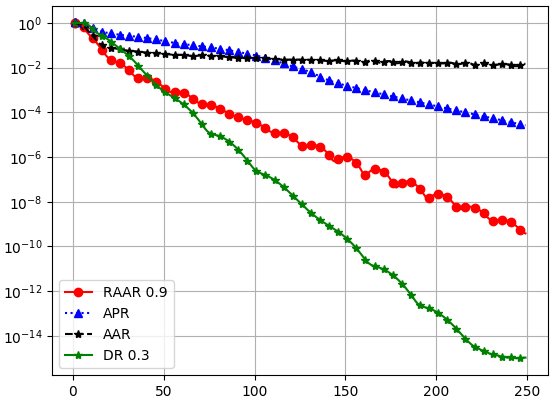}}
\subfigure[$\rho=0.1,\beta=0.9$]{\includegraphics[width=7cm]{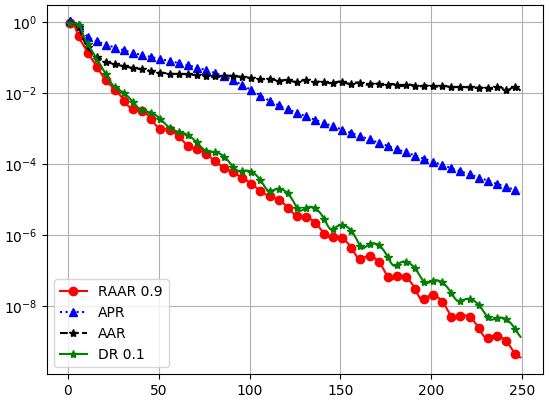}}
 \caption{Reconstruction (relative) error vs. iteration by various methods indicated in the legend with random initialization. 
 The straight-line feature (in all but AAR) in the semi-log plot indicates geometric convergence. 
 }
 \label{fig:raar}
 \end{figure}
 
The other,   the Relaxed AAR (RAAR),  is one of the best performing  phase retrieval algorithms  defined by the map
 \beq
 \label{raar}
 u_{k+1} &=& \beta\Gamma_0 (u_k)+(1-\beta)P_Y u_k,\quad \beta\in [\half,1],
 \eeq
 where $\Gamma_0$ is the Gaussian-DRS map with $\rho=0$ (i.e. AAR). 
 RAAR becomes AAR for $\beta=1$ (obviously) and AP for $\beta=\half$ (after some calculation) \cite{March16,Luke,Luke2}.

After some rearrangement  the fixed point equation for RAAR can be written as
\beqn
P_X x+P_X^\perp x=\beta P_X^\perp x+(P_X+(1-2\beta)P_X^\perp) P_Y x
\eeqn
from which it follows that 
\beqn
\px x= \px\py x, && \pxp x=- \lt({2\beta-1\over 1-\beta}\rt) \pxp\py x
\eeqn
and hence
\beq
\px x-\lt({1-\beta\over 2\beta-1}\rt) \pxp x&=& \px\py x+\pxp\py x=\py x.\label{411}
\eeq
Notably this is exactly the same fixed point equation for Gaussian-DRS with the corresponding parameter
\beq
\rho&=& {1-\beta\over 2\beta-1} \in [0,\infty) \label{beta}
\eeq
which tends to 0 and $\infty$ as $\beta$ tends to 1 and $\half$, respectively.
According to \cite{raar17} the optimal $\beta$ is usually between 0.8 and 0.9,  corresponding to $\rho=0.125$ and $0.333$ according to \eqref{beta}. We set  $\beta=0.9$ in Figure \ref{fig:raar}.

 In the experiments, we consider the setting of non-ptychographic phase retrieval with two coded diffraction patterns, one is the plane wave ($\mu=1$) and the other is $\mu=\exp(\im \theta)$ where $ \theta$  is independent  and uniformly distributed over $[0,2\pi)$. Theory of uniqueness of solution, up to a constant phase factor, is given in \cite{unique}.

 Figure \ref{fig:raar} shows the relative error (modulo a constant phase factor) versus iteration of RAAR ($\beta=0.9$ round-bullet solid line), APR (blue-triangle dotted line), AAR
 (black-star dashed line) and Gaussian-DRS with (a) $\rho=1.1, $ (b) $\rho=0.5,$ (c) $\rho=0.3$ and (d) $\rho=0.1$. 
 Note that the AAR, APR and RAAR lines vary slightly across different plots because of random initialization. 
 
 The straight-line feature (in all but AAR) in the semi-log plot indicates global geometric convergence. 
 The case with AAR is less clear in Figure \ref{fig:raar}. But it has been shown that the AAR sequence converges geometrically near the true object  (after applying $A^+$) but converges in power-law ($ \sim k^{-\alpha}$ with $\alpha\in [1,2]$) from random initialization \cite{FDR}.

 Figure \ref{fig:raar} shows  that  APR outperforms  AAR (consistent with the prediction of Corollary \ref{cor:rate1}) but underperforms RAAR. By decreasing $\rho$ to either $0.5$ or $0.1$, the performance of Gaussian-DRS closely matches that of RAAR. The optimal parameter
 appears to lie between  $0.1$ and $0.5$. For example, with $\rho=0.3,$ Gaussian-DRS significantly outperforms RAAR.  
The oscillatory behavior of Gaussian-DRS in (d) is due to the dominant complex eigenvalue of $J_A$.

}

\section{Blind ptychography algorithm}\label{sec:blind}

In the next two sections we apply the DRS methods to the more challenging problem of blind ptychography. 
In blind ptychography, we do not assume the full knowledge of the probe which is to be recovered simultaneously with the unknown object. 

Let $\nu^{0}$ and $g=\vee_\bt g^\bt$ be any pair
 of the probe and the object estimates producing  the same ptychography data as $\mu^{0}$ and $f$, i.e.
 the diffraction pattern of $\nu^\bt\odot g^\bt$ is identical to that of $\mu^\bt\odot f^\bt$ where
 $\nu^\bt$ is the $\bt$-shift of $\nu^{0}$ and $g^\bt$ is the restriction of $g$ to $\cM^\bt$. 
We refer to the pair $(\nu^0,g)$ as   a  blind-ptychographic solution (in the object domain) and
$(\mu^0,f)$ as the true
solution.

We can write  the total measurement data as $b=|\cF(\mu^0,f)|$ where $\cF$ is  the concatenated  oversampled Fourier transform acting on
$\{\mu^\bt\odot f^\bt:\bt \in \cT\}$  (see Appendix \ref{app:matrix}), i.e. a bi-linear transformation in the direct product of the probe space and the object space.  By definition, a blind-ptychographic solution $(\nu^0,g)$ satisfies  $|\cF(\nu^0, g)|=b$.

There are two ambiguities inherent to any blind ptychography.
 
The first is the affine phase ambiguity. Consider the probe and object estimates
\beq
\label{lp1}
\nu^{0}(\bn)&=&\mu^{0}(\bn) \exp(-\im a -\im \bw\cdot\bn),\quad\bn\in\cM^{0}\\
\label{lp2} g(\bn)&=& f(\bn) \exp(\im b+\im \bw\cdot\bn),\quad\bn\in \IZ^2_n
\eeq
for any $a,b\in \IR$ and $\bw\in \IR^2$.  For any $\bt$, we have the following
calculation
\beqn
\nu^\bt(\bn)&=&\nu^{0}(\bn-\bt)\\
&=&\mu^{0}(\bn-\bt) \exp(-\im \bw\cdot(\bn-\bt))\exp(-\im a)\\
&=&\mu^\bt(\bn) \exp(-\im \bw\cdot(\bn-\bt))\exp(-\im a)
\eeqn
and hence for all $\bn\in \cM^\bt, \bt\in\cT$
\beq
\label{drift2}
\nu^\bt(\bn) g^\bt(\bn)&=&\mu^\bt(\bn)f^\bt(\bn) \exp(\im(b-a))\exp(\im \bw\cdot\bt). 
\eeq
Clearly, \eqref{drift2} implies that $g$ and $\nu^{0}$ produce the same ptychographic data as $f$ and $\mu^{0}$ since
for each $\bt$, $\nu^\bt\odot g^\bt$ is a constant phase factor times $\mu^\bt\odot f^\bt$ {  where $\odot$ is the entry-wise (Hadamard) product}.  It is also clear that the above statement holds
true regardless of the set $\cT$ of shifts and the type of probe. 

In addition to the affine phase ambiguity \eqref{lp1}-\eqref{lp2},  a scaling factor ($g=c f, \nu^{0}=c^{-1} \mu^{0}, c>0$) is inherent to any blind ptychography. 
Note that when the probe is exactly known (i.e. $\nu^{0}=\mu^{0}$), neither ambiguity can occur.

Besides the inherent ambiguities,  blind ptychography imposes extra demands on the scanning
scheme. For example, there are many other ambiguities inherent to the regular raster scan:  $\cT=\{\bt_{kl}={\tau}(k,l): k, l\in \IZ\}$ unless the step size $\tau=1$.
Blind ptychography with a raster scan produces $\tau$-periodic  ambiguities called the raster scan pathology as well as
non-periodic ambiguities associated with block phase drift. The reader to referred to
\cite{raster} for a complete analysis of ambiguities associated with the raster scan.

A conceptually simple (though not necessarily the most practical) way to remove these ambiguities is  introducing small irregular perturbations to 
the raster scan with $\tau>m/2$, i.e. the overlap ratio greater than 50\% (see \eqref{rank1} and \eqref{rank2}).  For a thorough analysis of the conditions for blind ptychography,
we refer the reader to \cite{blind-ptych, raster}.

The basic strategy for blind ptychographic reconstruction is to alternately update the object and probe estimates
starting from an initial guess as outlined in Algorithm 1 \cite{DM08,probe09,pum}.  
\begin{algorithm}
\caption{Alternating minimization (AM)}\label{alg: suedo algorithm}
\begin{algorithmic}[1]
\State Input: initial probe guess $\mathbf{\mu}_1$ and object guess $f_1$. 
\State Update the object estimate  
$\quad
f_{k+1}=\arg\min L(A_kg)$ s.t. $g\in \IC^{n\times n}$.
\State  Update the probe estimate  
$\quad
\mu_{k+1}=\arg\min L(B_k\nu)$ s.t.
$\nu\in \IC^{m\times m}$. 
\State Terminate if $\||B_k\mu_{k+1}|- b\| $ stagnates or is less than tolerance; otherwise, go back to step 2 with $k\rightarrow k+1.$
\end{algorithmic}
\end{algorithm}

We solve the inner loops (step 2 and 3 in Algorithm 1) and update the object and probe estimates 
by the DRS methods where $A_k h:=\cF(\mu_k, h), \forall h\in \IC^{n^2}, $ defines a matrix $A_k$  for the $k$-th probe estimate $\mu_k$ and $B_k\eta :=\cF(\eta, f_{k+1}), \forall \eta\in \IC^{m^2},$  for  the $(k+1)$-st image estimate $f_{k+1}$.  

For ease of reference, we denote
Algorithm 1 with Gaussian-DRS and Poisson-DRS by Gaussian-DRSAM and Poisson-DRSAM, respectively. 

\subsection{Initialization}\label{sec:ppc}
 For non-convex iterative optimization, a good initial guess  or some regularization is usually crucial for convergence \cite{ML12,Poisson2}. The initialization step is often glossed over in the development of numerical schemes. This is even more so for blind ptychography which is doubly non-convex because, in addition to the phase retrieval step, extracting the probe and the object from their product is also non-convex.

We say  that a probe estimate $\nu^0$ satisfies PPC$(\delta)$ (standing for the probe phase constraint) if 
\beq\label{PPC}
\measuredangle (\nu^0(\bn),\mu^0(\bn))<\delta\pi,\quad \forall\bn 
\eeq
where  $\delta\in (0, 1/2]$ is the uncertainty parameter. 

PPC$(\delta)$ defines an alternative measure to the standard norm-based metric. 
Our default case is $\delta=0.5$ with which  PPC  is equivalent to 
$\Re(\bar \nu^0\odot \mu^0)>0$ (where the bar denotes the complex conjugate) has the intuitive meaning that at every pixel $\nu^0$ and  $\mu^0$ point to the same half plane in $\IC$ (Figure \ref{fig:MPC}). 

 \begin{figure}
 \centering
 \includegraphics[width=7cm]{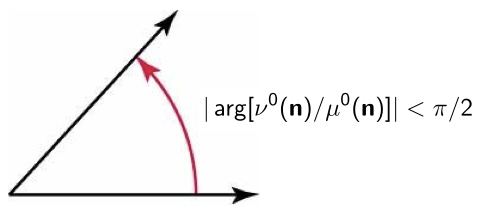}
\caption{$\nu^0$ satisfies MPC if $\nu_0(\bn)$ and $\mu^0(\bn)$ form an acute angle for all $\bn$.}
\label{fig:MPC}
\end{figure}

Our initialization method is inspired by the uniqueness theory in \cite{blind-ptych} which proves PPC (0.5) is required to remove
all other ambiguities than the inherent ones (the affine phase factor and the constant scaling factor).

Under PPC,  however, the initial probe  may be  significantly far away from the true probe in
norm. Even if  $|\mu_1(\bn)|=|\mu^{0}(\bn)|=\mbox{const.}$, 
 the probe guess with uniformly distributed $\phi$ in $(-\pi/2, \pi/2]$  has the relative error close to
\[
\sqrt{{1\over \pi}\int^{\pi/2}_{-\pi/2} |e^{\im\phi}-1|^2 d\phi}=\sqrt{2(1-{2\over \pi})}\approx   0.8525 
\]
with high probability. 
We use \eqref{PPC} for selecting and quantifying initialization, instead of the usual 2-norm.  
Non-blind ptychography gives rise to infinitesimally small $\delta$. In practice, \eqref{PPC} needs only to hold for sufficiently large number of pixels $\bn$.  

In summary, 
in our numerical experiments we use the following probe initialization denoted by PPC
\beq
\label{probe-initial}
\mu_1 (\mathbf{n})=\mu^0(\mathbf{n})\,\exp{\lt[\im 2\pi\frac{ \mathbf{k}\cdot\mathbf{n}}{n}\rt]}\,\exp{[\im \phi(\mathbf{n})]},\ \ \ \mathbf{n}\in \mathcal{M}^0
\eeq
where $\phi(\bn)$ are independently and uniformly distributed on $(-\pi/2, \pi/2)$. 
In our numerical experiments, PPC  results in geometric convergence for any $\bk$ (even though the limiting solution may end up with a different $\bk$ as allowed  by linear phase ambiguity).

\section{Numerical experiments for blind ptychography}
\label{sec:num}

 We test the DRS methods with $\rho=1$ for blind ptychography and demonstrate that even with this far from optimal parameter (cf. Corollary \ref{cor1} and Figure \ref{fig:raar}), DRSAM converges geometrically under  the nearly minimum conditions established in the uniqueness theory \cite{blind-ptych} (see also Section \ref{sec:ppc} and Section \ref{sec:scan}). 

The inner loops of Gaussian DRSAM become
\beqn
u_{k}^{l+1} &= &\frac{1}{2}u_{k}^{l} +\frac{1}{2}b\odot \sgn\big(R_ku_{k}^{l}\big)\\
v_{k}^{l+1}& =& \frac{1}{2} v_{k}^{l}+\frac{1}{2}b\odot \sgn{\Big(S_k v_{k}^{l}\Big)}. 
\eeqn
and the inner loops of the Poisson DRSAM become
 \beq\label{770} u_{k}^{l+1}& =& \frac{1}{2}u_{k}^{l} -\frac{1}{3}R_k u_{k}^{l}+\frac{1}{6} \sgn{\Big(R_k u_{k}^{l}\Big)}\odot \sqrt{|R_k u^l_k|^2+24b^2} \\
 v_{k}^{l+1} &=& \frac{1}{2} v_{k}^{l} -\frac{1}{3}S_k v_{k}^{l}+\frac{1}{6}\sgn{\Big(S_k v_{k}^{l}\Big)}\odot \sqrt{|S_k v_{k}^{l}|^2+24b^2}.\label{771}
\eeq
Here 
$R_k=2P_k-I$ is the  reflector corresponding to the projector $P_k:=A_kA_k^+ $ and $S_k$ is the reflector corresponding to the projector  $Q_k:=B_kB_k^+ $. 
We set  $u^1_k=u^{\infty}_{k-1}$ where $u^{\infty}_{k-1}$ is the terminal value
at epoch $k-1$ and $v^1_k=v^{\infty}_{k-1}$ where $v^{\infty}_{k-1}$ is the terminal value
at epoch $k-1$.

 \subsection{Test objects}
\begin{figure}[t]
\centering
 \subfigure[RPP magnitudes]{\includegraphics[width=6.5cm]{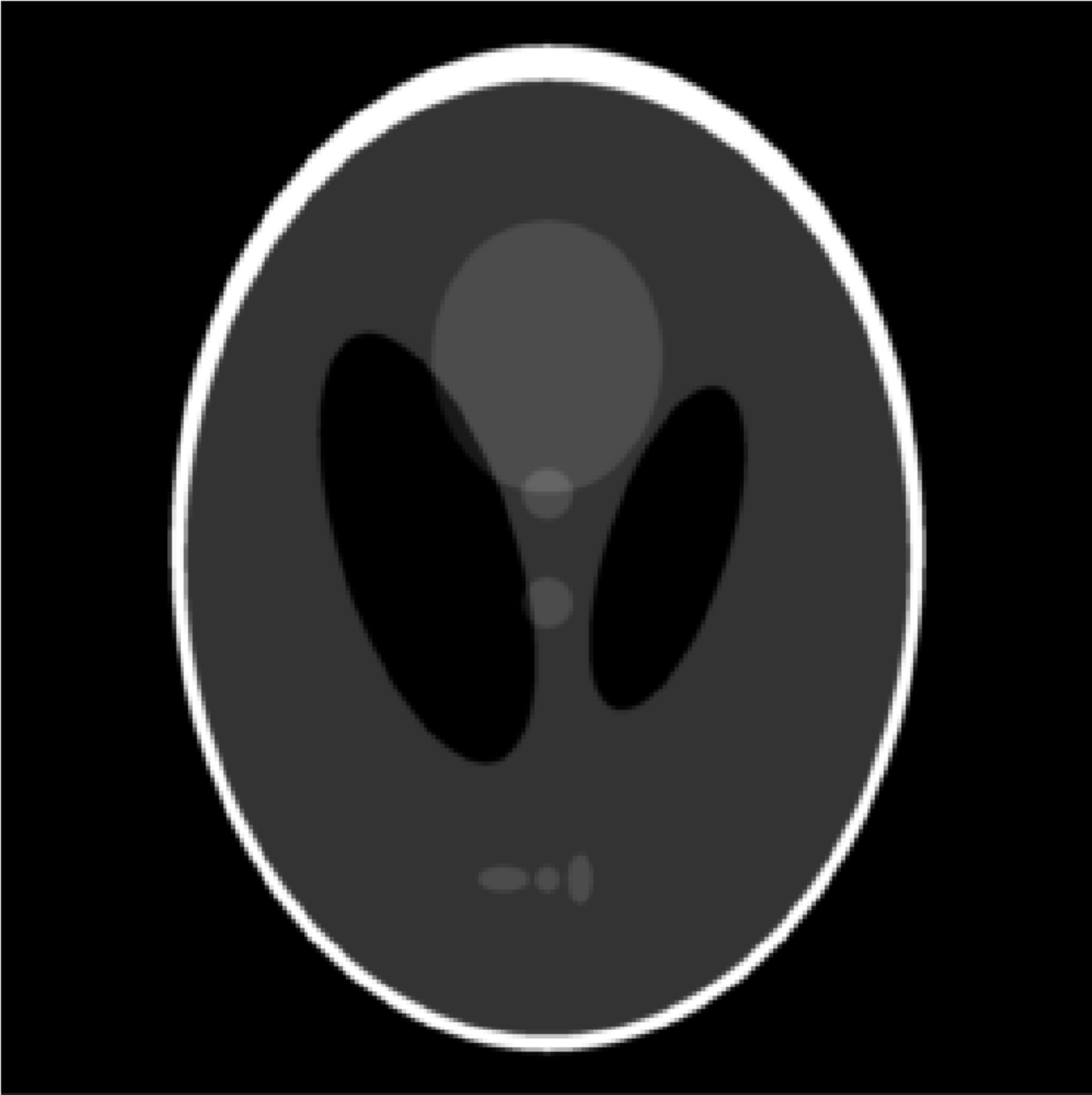}}\hspace{0.7cm}
\subfigure[RPP phases ]{\includegraphics[width=7cm,height=6.6cm]{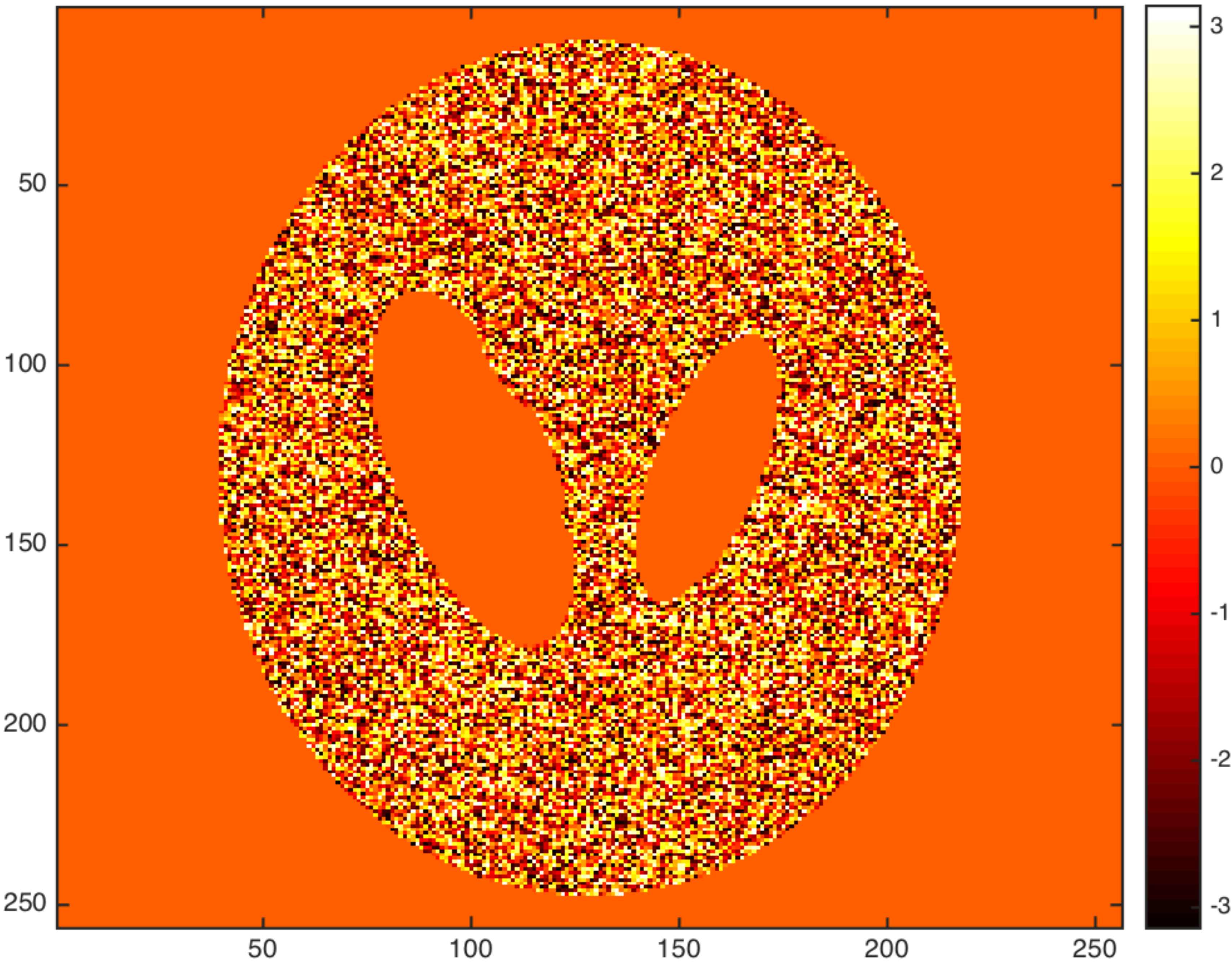}}  \\
 \caption{(a) Magnitudes and (b) phases of  RPP. }
 \label{fig:RPP}
 \end{figure}

In addition to CiB, 
our second test object is randomly-phased phantom (RPP) defined by $
f=P\odot  e^{\im\phi}
$
where $P$ is the standard phantom (Fig. \ref{fig:RPP}(a)) and  $\{\phi(\bn)\}$ are i.i.d. uniform random variables over $[0,2\pi]$. 
RPP has the maximal phase range because of its noise-like phase profile.  In addition to the huge phase range, RPP has loosely supported parts with respect to the measurement schemes (see below) due to its thick dark margins around the oval. 

The third test object is the salted RPP, the sum of
RPP and the salt noise (not shown). The salted noise is i.i.d, binomial random variables  with  probability $0.02$ to be a complex constant in the form of $a(1+\im), a\in \IR, $ and probability $0.98$ to be zero.  The salt noise reduces  the support looseness  without significantly changing the original image making  the salted RPP more connected with respect to the ptychographic measurement.  
\subsection{Probe function}

We use 
a randomly phased probe with the unknown transmission function  
$
\mu^0(\bn)=e^{\im \theta(\bn)}
$
where  $\theta(\bn)$ are random variables.   Randomly phased probes have been adopted in ptychographic experiments
\cite{zone-plate,ptycho-rpi,random-aperture,diffuser}.

\begin{figure}
\centering
\subfigure[i.i.d. probe]{\includegraphics[width=7cm]{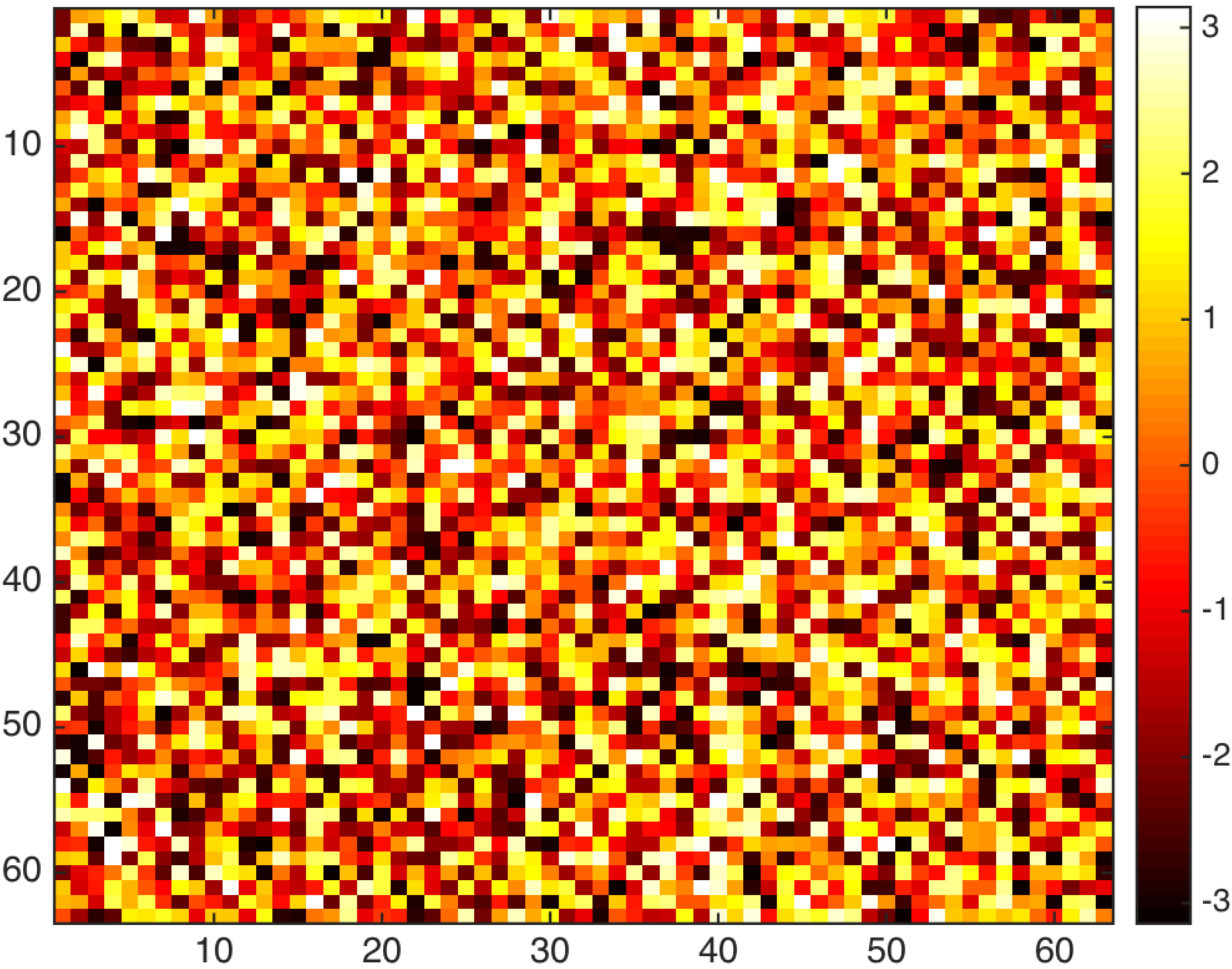}}\quad
\subfigure[Correlated probe $c=0.4$]{\includegraphics[width=7cm]{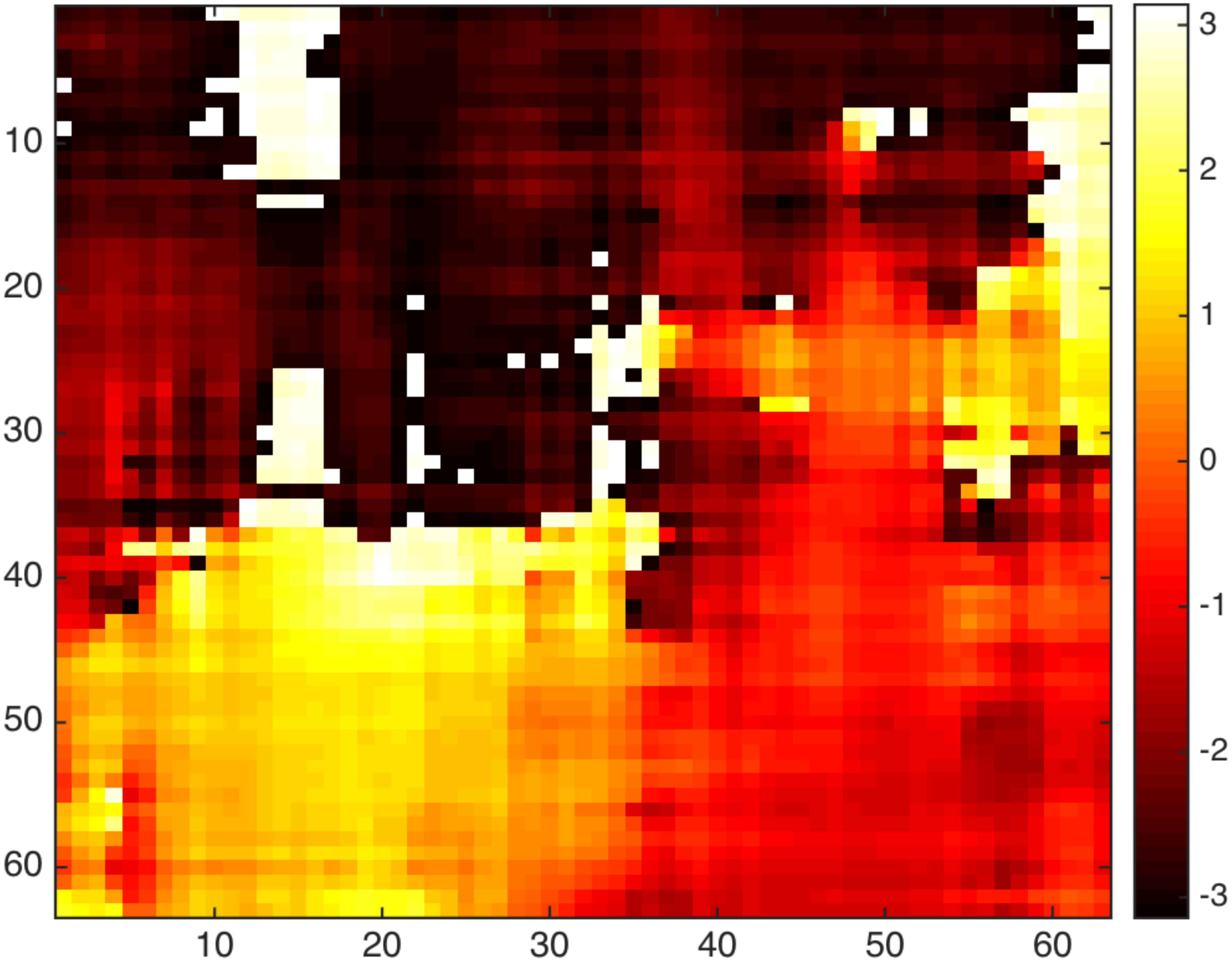}}\\
\subfigure[Correlated probe $c=0.7$]{\includegraphics[width=7cm]{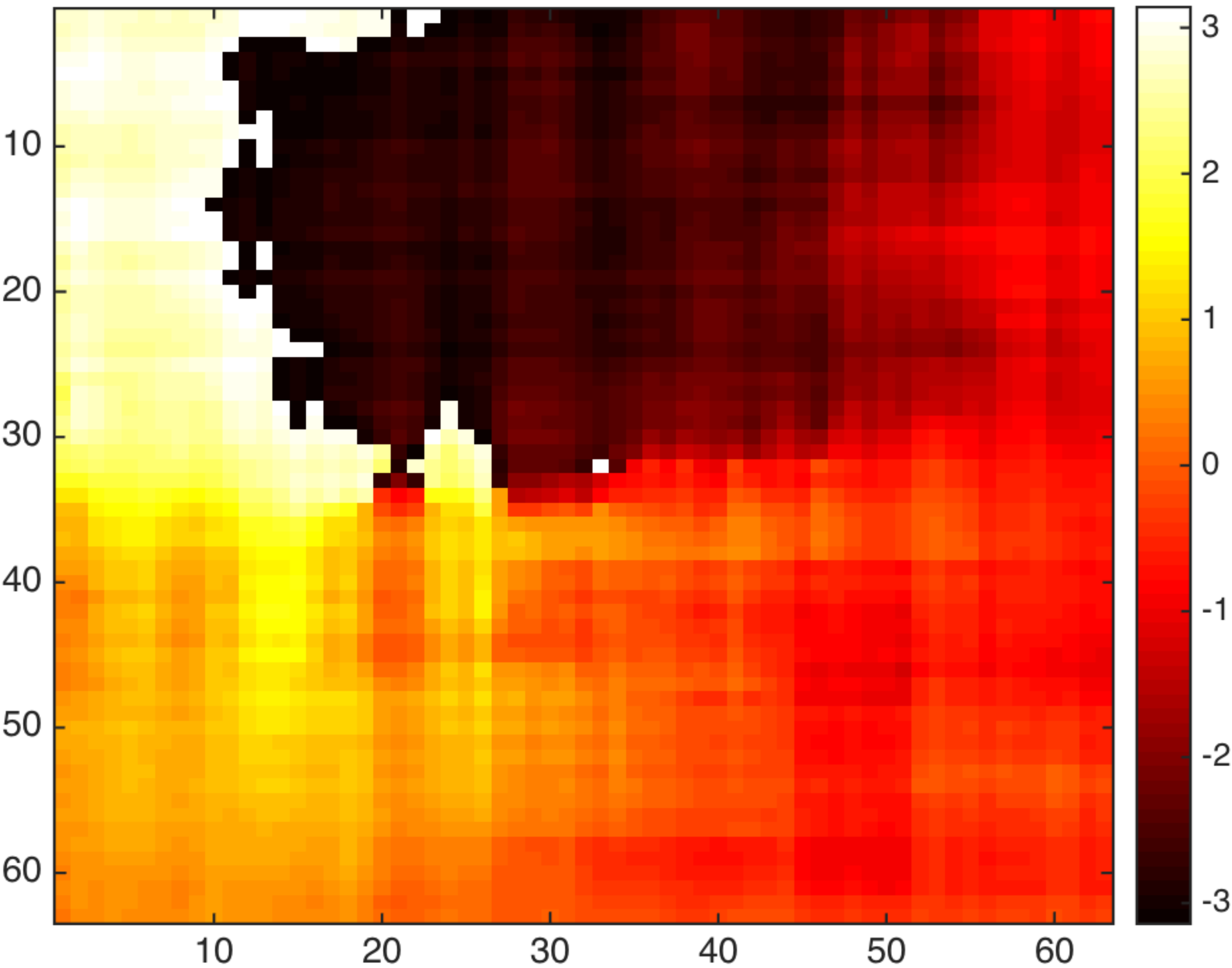}}\quad
\subfigure[Correlated probe $c=1$]{\includegraphics[width=7cm]{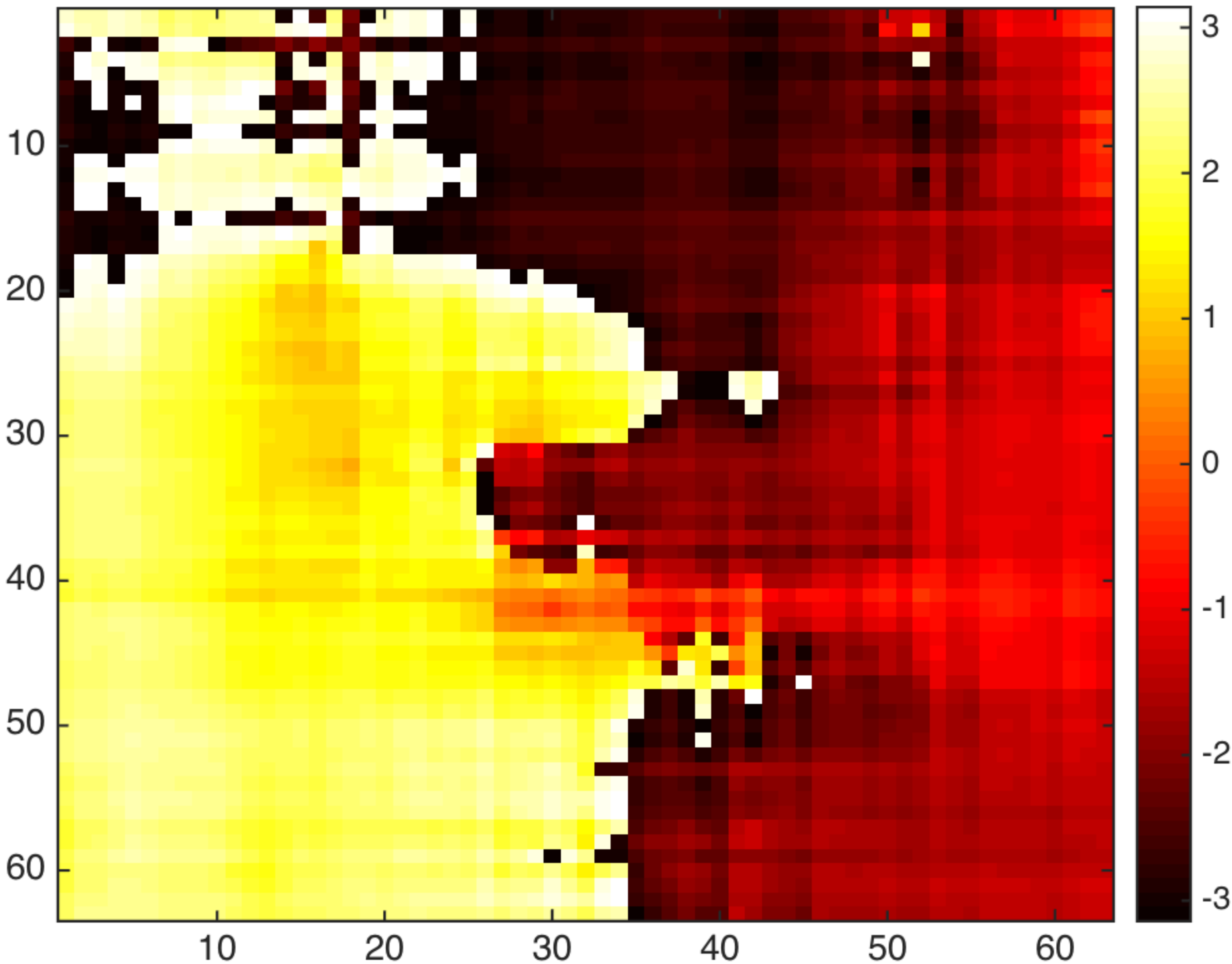}}
\caption{The phase profile of (a) the i.i.d. probe and (b)(c)(d) the correlated probes of various correlation lengths.}
\label{fig:probes}
\end{figure}
We do not explore the issue of varying 
the probe size in the present work, which was carried out for AAR in \cite{ptych-unique}.
We fix the probe size to $60\times 60$. In addition to the i.i.d. probe, we consider also correlated probe
produced by convolving the i.i.d. probe with characteristic function of the set $\{(k_1,k_2 )\in \mathbb{Z}^2 : \max\{|k_1|,|k_2|\} \leq c\cdot m;\ c\in (0,1]  \}$ where  the constant $c$ is a measure
of the correlation length in the unit of $m=60$ (Fig. \ref{fig:probes}).

\subsection{Error metrics for blind ptychography}

 We use relative error (RE) and relative residual (RR) as the merit metrics  for the recovered image $f_k$ and probe $\mu_k$ at the $k^{th}$ epoch:
\beq\label{RE}
\mbox{RE}(k)&= &\min_{\alpha\in \mathbb{C}, \mathbf{r} \in \mathbb{R}^2}\frac{\sqrt{\sum_\bn|f(\bn) - \alpha e^{-\im {2\pi}\mathbf{n}\cdot \mathbf{r}/n} f_k(\bn)|^2}}{\|f\|}\\
\mbox{\rm RR}(k)&= & \frac{\|b- |A_kf_{k}|\|}{\|b\|}.
\eeq
Note that in \eqref{RE} both the affine phase and the scaling factors are discounted. 

\subsection{Sampling schemes} \label{sec:scan}
 \begin{figure}
\centering
\subfigure[Perturbed grid given by \eqref{rank1}]{\includegraphics[width=6cm,height=5.7cm]{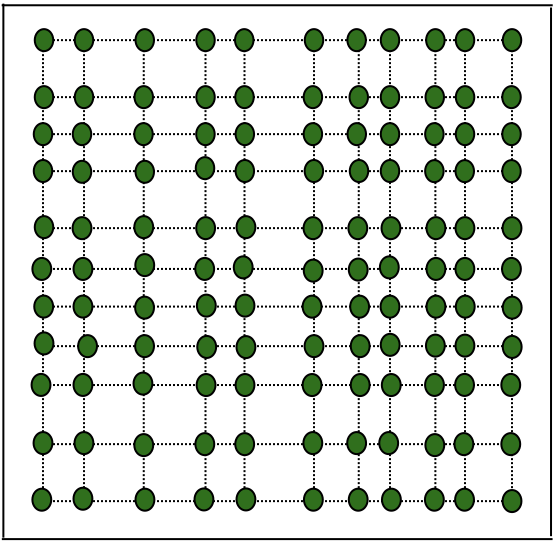}}\hspace{1cm}
\subfigure[Perturbed grid given by \eqref{rank2}]{\includegraphics[width=6cm,height=5.7cm]{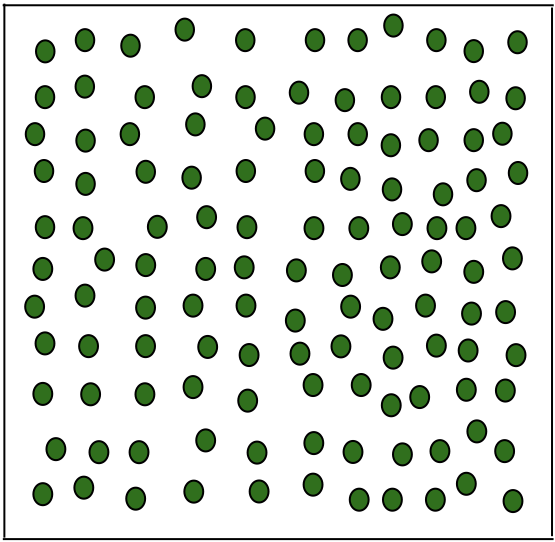}}
\caption{Two perturbed raster scans}
\label{fig:irregular}
\end{figure}

The uniqueness theorem for blind ptychography \cite{raster} holds for the following irregularly perturbed  raster scans
\beq
\label{rank1}
\mbox{\rm Rank-one perturbation}\quad \bt_{kl}=\tau(k,l)+(\delta^1_k, \delta^2_l),\quad k,l\in \IZ
\eeq
where $\delta_{k}^1$ and $\delta^2_l$ are small random variables relative to $\tau$. 
The other is 
\beq
\label{rank2}
\mbox{\rm Full-rank perturbation}\quad \bt_{kl}=\tau (k,l)+(\delta^1_{kl},\delta^2_{kl}), \quad k,l\in \IZ
\eeq
 where $\delta_{kl}^1$ and $\delta^2_{kl}$ are small random variables relative to $\tau$. 
 Here the stepsize $\tau<m/2$ corresponding to the overlap ratio greater than $50\%$. 
 The $50\%$ overlap ratio has been proved to be a nearly minimum requirement for uniqueness with the perturbed raster scans.
 
 We let  $\delta_{k}^1$ and $\delta^2_l$ in the rank-one scheme \eqref{rank1} and 
 $\delta_{kl}^1$ and $\delta^2_{kl}$ in the full-rank scheme \eqref{rank2} to be {i.i.d. uniform random variables  over $\lb -4,4\rb$}. 
In other words, the adjacent probes overlap by an average of $\tau/m=50\%$.

\subsection{Different combinations}
First we compare performance of DRSAM with different combinations of loss functions, scanning schemes and
random probes in the case of  noiseless measurements with  the periodic boundary condition. 
We use the stopping  criteria for the inner loops: 
\begin{equation*}
\frac{\| |P_k u_{k}^{l}|-b\|-\||P_k u_{k}^{l+1}|-b\|}{\||P_k u_{k}^{l}|-b\|} \leq 10^{-4}
\end{equation*} 
with the maximum number of iterations capped at 60.

\begin{figure}
\centering 
\subfigure[]{\includegraphics[width=7cm]{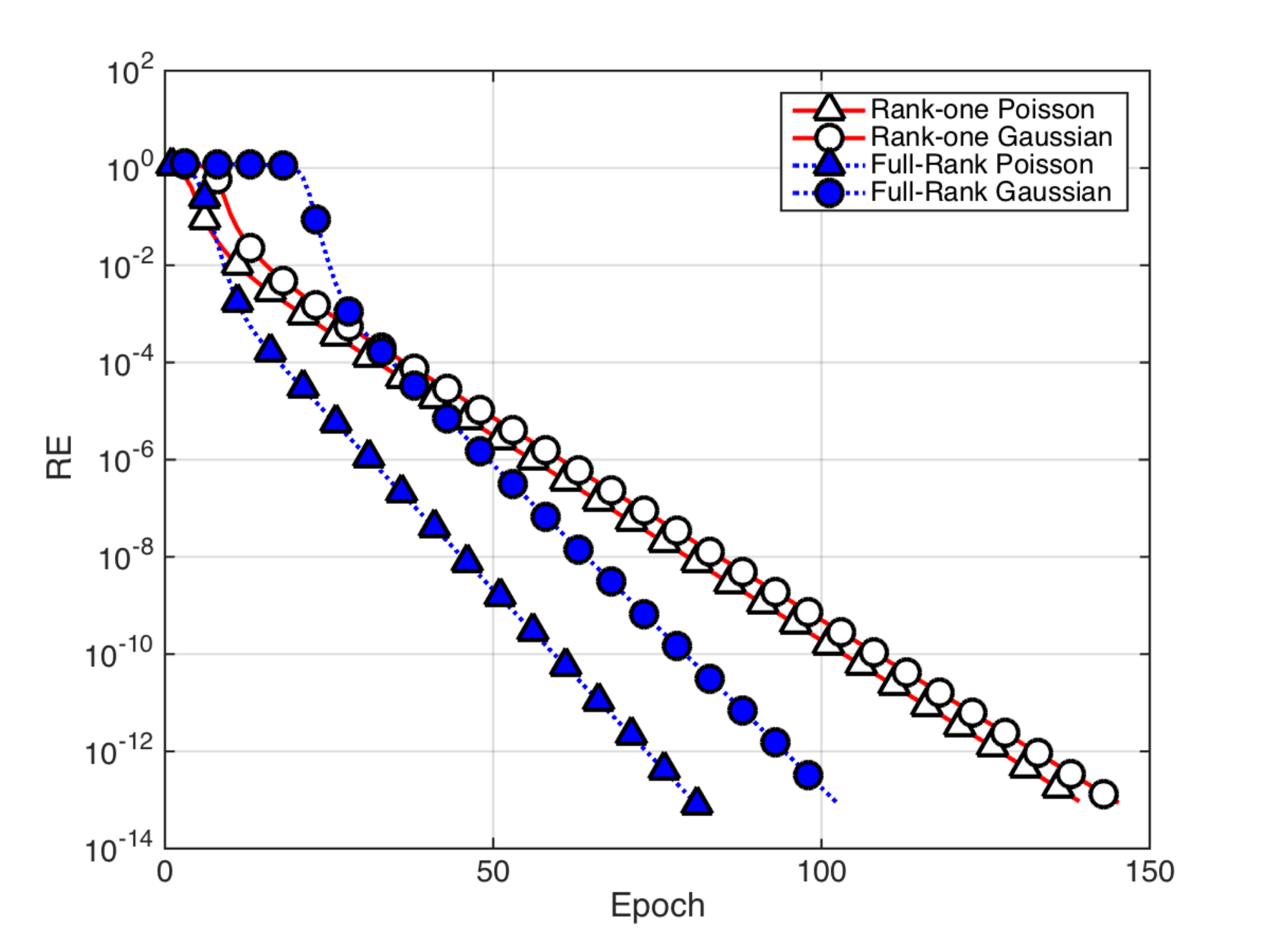}}
\subfigure[]{\includegraphics[width=7cm]{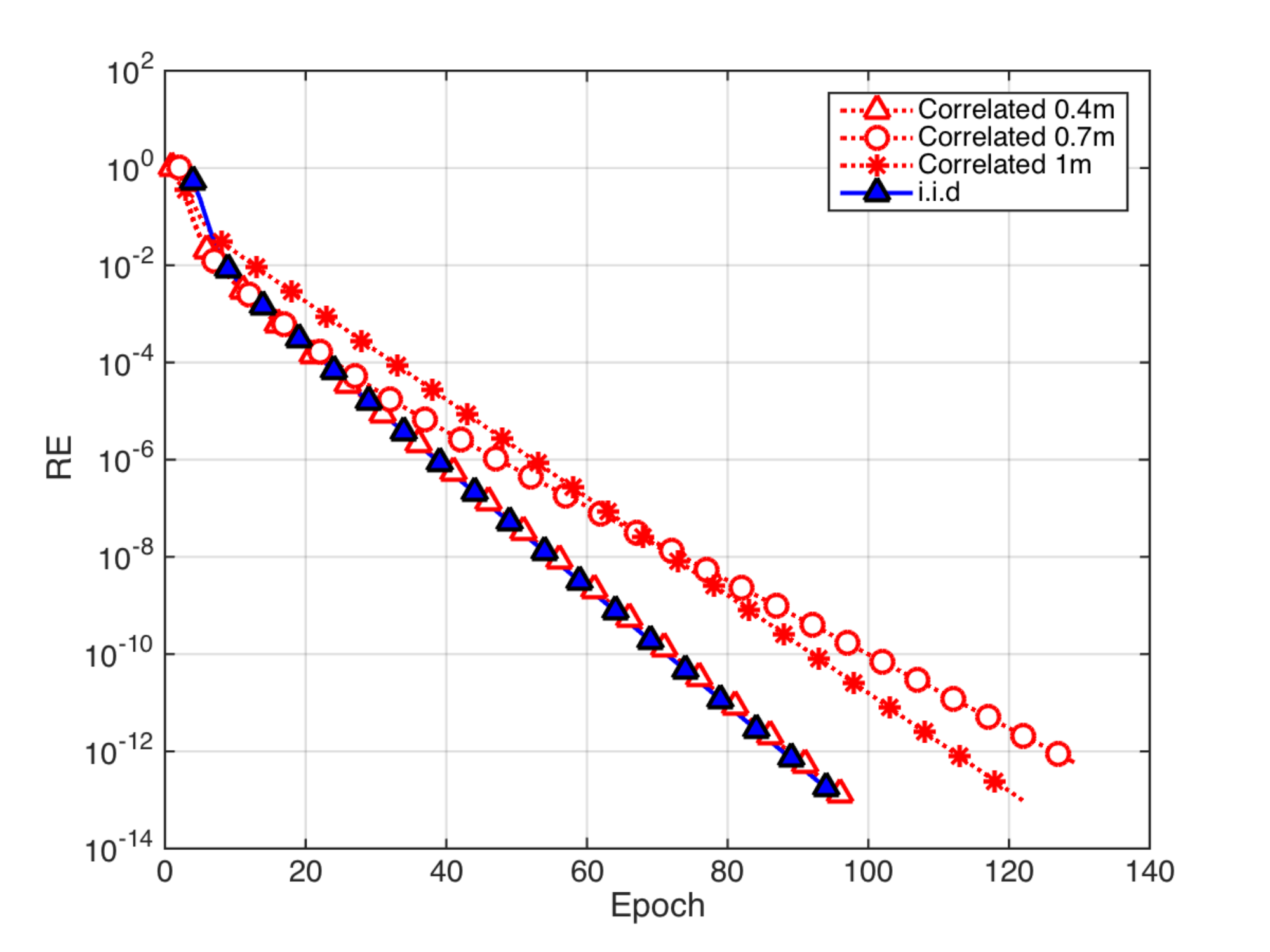}} \\
  \caption{ Geometric convergence to CiB at various rates for (a) Four combinations of loss functions and scanning schemes with i.i.d. probe (rank-one Poisson, $\mbox{rate}=0.8236$; rank-one Gaussian, $\mbox{rate}=0.8258$; full-rank Poisson, $\mbox{rate}=0.7205$; full-rank Gaussian, $\mbox{rate}=0.7373$) and (b) Poisson-DRS with four probes of different correlation lengths ($\mbox{rate}=0.7583$ for $c=0.4$; $\mbox{rate}=0.8394$ for  $c=0.7$; $\mbox{rate}=0.7932$ for $c=1$; $\mbox{rate}=0.7562$ for iid probe)
  }
  \label{fig1}
  \label{fig2}
\end{figure}%

Figure \ref{fig1} shows geometric decay of RE \eqref{RE}  at various rates for the test object CiB. 
In particular, Fig. \ref{fig1}(a) shows that the full-rank scheme outperforms the rank-one scheme
and that Poisson-DRS outperforms (slightly) Gaussian-DRS while
Figure \ref{fig2}(b) shows that the i.i.d. probe yields the smallest rate of convergence ($=0.7562$) closely followed by
the rate ($=0.7583$)  for $c=0.4$.

\subsection{Boundary conditions}

 \begin{figure}
\centering
 \subfigure[Max iteration for inner loops = 80]{\includegraphics[width=7cm]{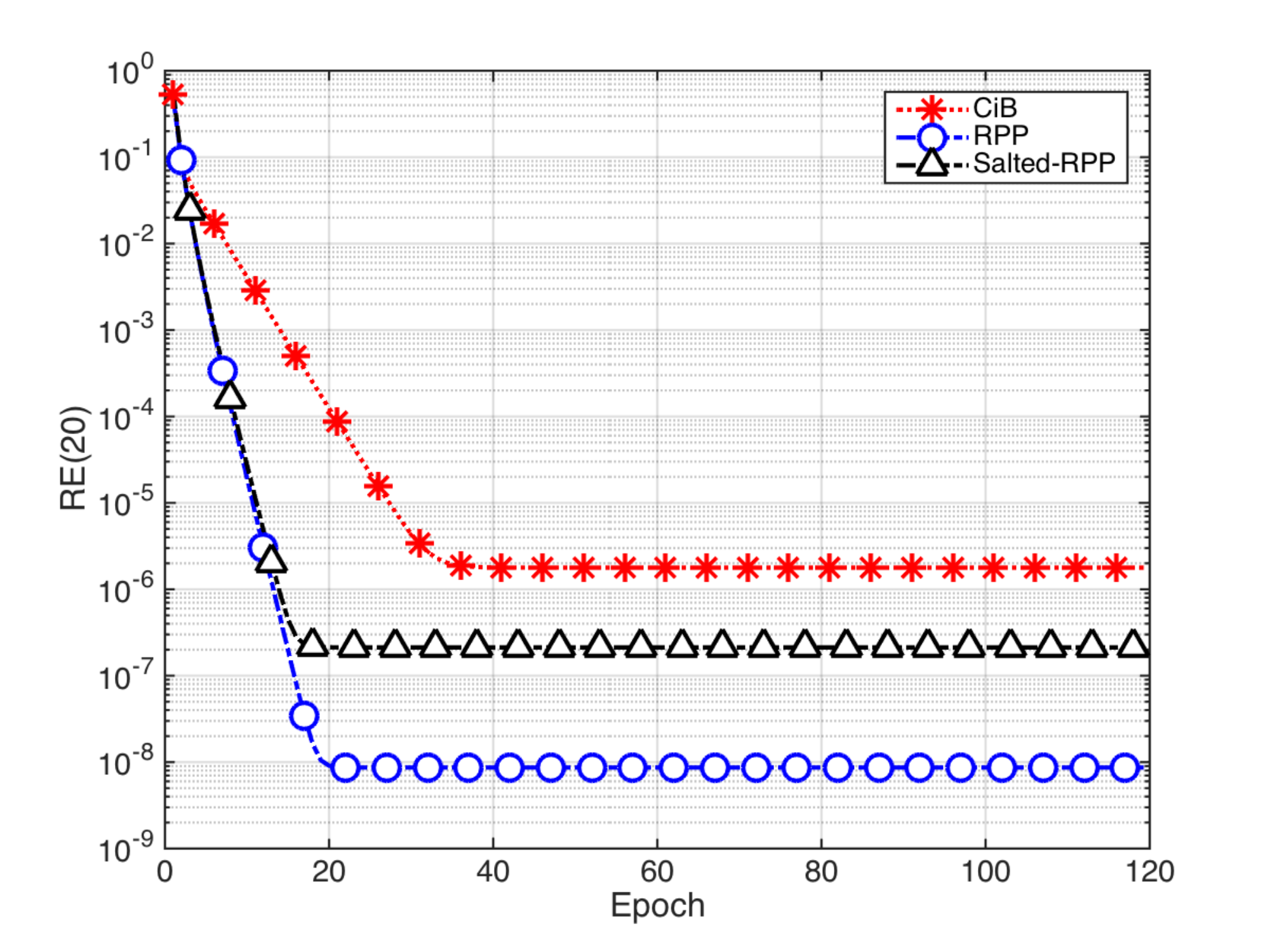}}
 \subfigure[Max iteration for inner loops = 110]{\includegraphics[width=7cm]{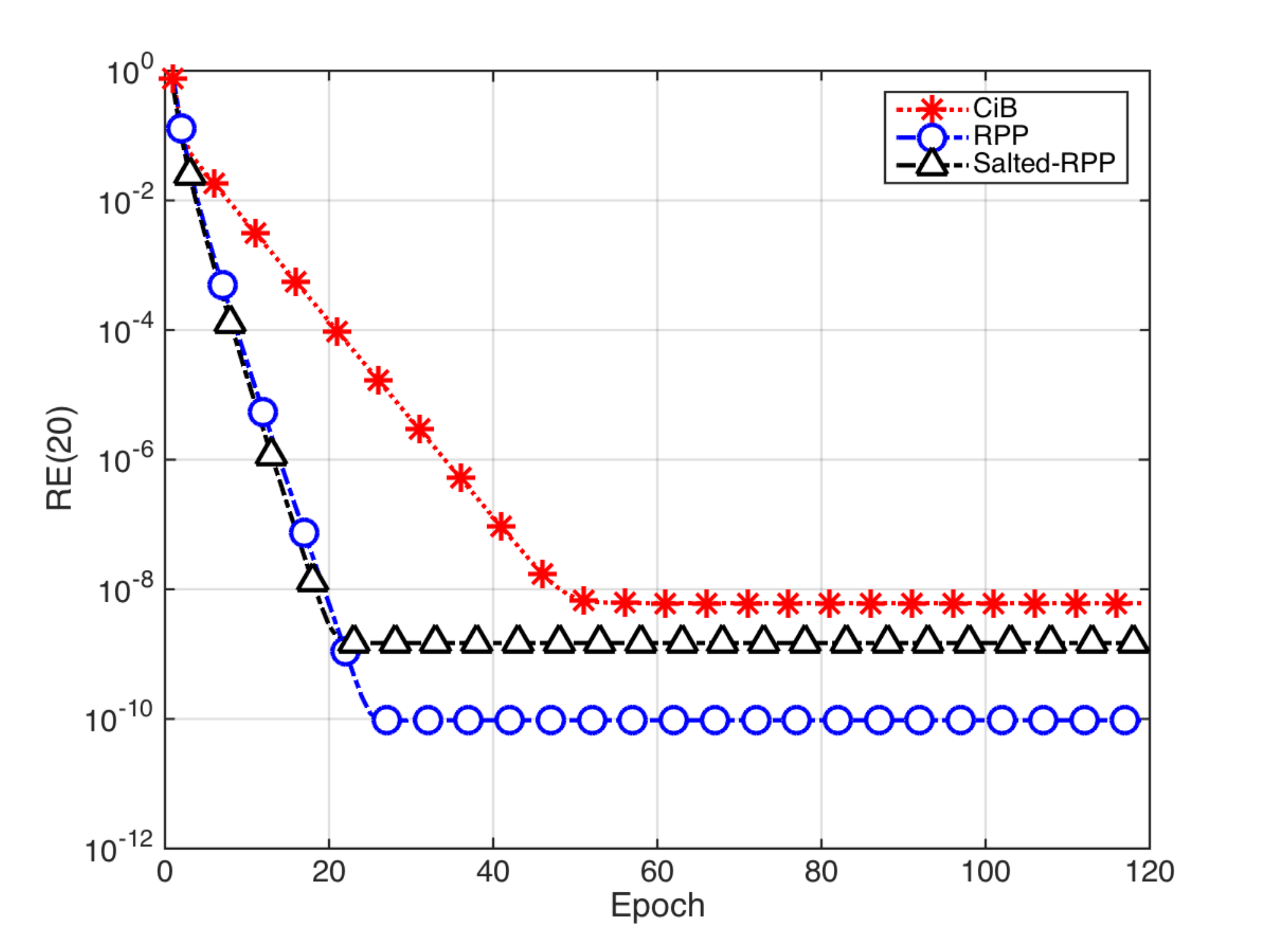}}
\caption{RE2 under  the bright-field condition = 255. }
\label{fig:bright}
\end{figure}

The periodic boundary condition conveniently  treats all diffraction patterns and object pixels
in the same way by assuming that $\IZ_n^2$ is a (discrete) torus. The periodic boundary condition generally forces 
the slope $\br$ in the affine phase ambiguity to be integers.
For 3D blind tomography, however, different linear phase ramps  from different projections would collectively create enormous 3D ambiguities 
that are difficult to make consistent and hence it is highly desirable to remove the linear phase ambiguity early on in the process.

To this end, we consider the non-periodic bright-field  boundary conditions taking on some nonzero value in $\cM\setminus \IZ_n^2$. 
We aim to show that the affine phase ambiguity is absent under the bright-field boundary condition.

We test the Poisson-DRSAM with the full-rank scheme   with  a more stringent  error metric
\beq\label{RE2}
{\mbox{RE2}}(k)&= &\min_{\theta\in \IR}\frac{\| f -e^{\im\theta} f_k\|}{\|f\|}. 
\eeq

We also use the less tolerant  stopping rule
\begin{equation*}
\frac{\| |P_k u_{k}^{l}|-b\|-\||P_k u_{k}^{l+1}|-b\|}{\||P_k u_{k}^{l}|-b\|} \leq 10^{-5}
\end{equation*} 
for the inner loops with the maximum number of iteration  capped at 80.

Fig. \ref{fig:bright} demonstrates the capability of the bright-field boundary condition ($=255$) to eliminate the linear phase ambiguity
as the stronger error metric \eqref{RE2} decays geometrically before settling down to the final level of accuracy. 
 The final level of accuracy, however, depends on how accurately the inner loops for each epoch are solved. For example, increasing the maximum number of iteration from 80 (Figure \ref{fig:bright}(a))  to 110  (Figure \ref{fig:bright}(b)), significantly enhances the final accuracy
of reconstruction. 

We also see that the bright-field condition enforcement has a better result on RPP than CiB. 

\subsection{Comparison with rPIE}

\begin{figure}[t]
\centering
\subfigure[rPIE w. PPC(0.025) \commentout{($0,0,\frac{1}{40}$)}]{\includegraphics[width=7cm]{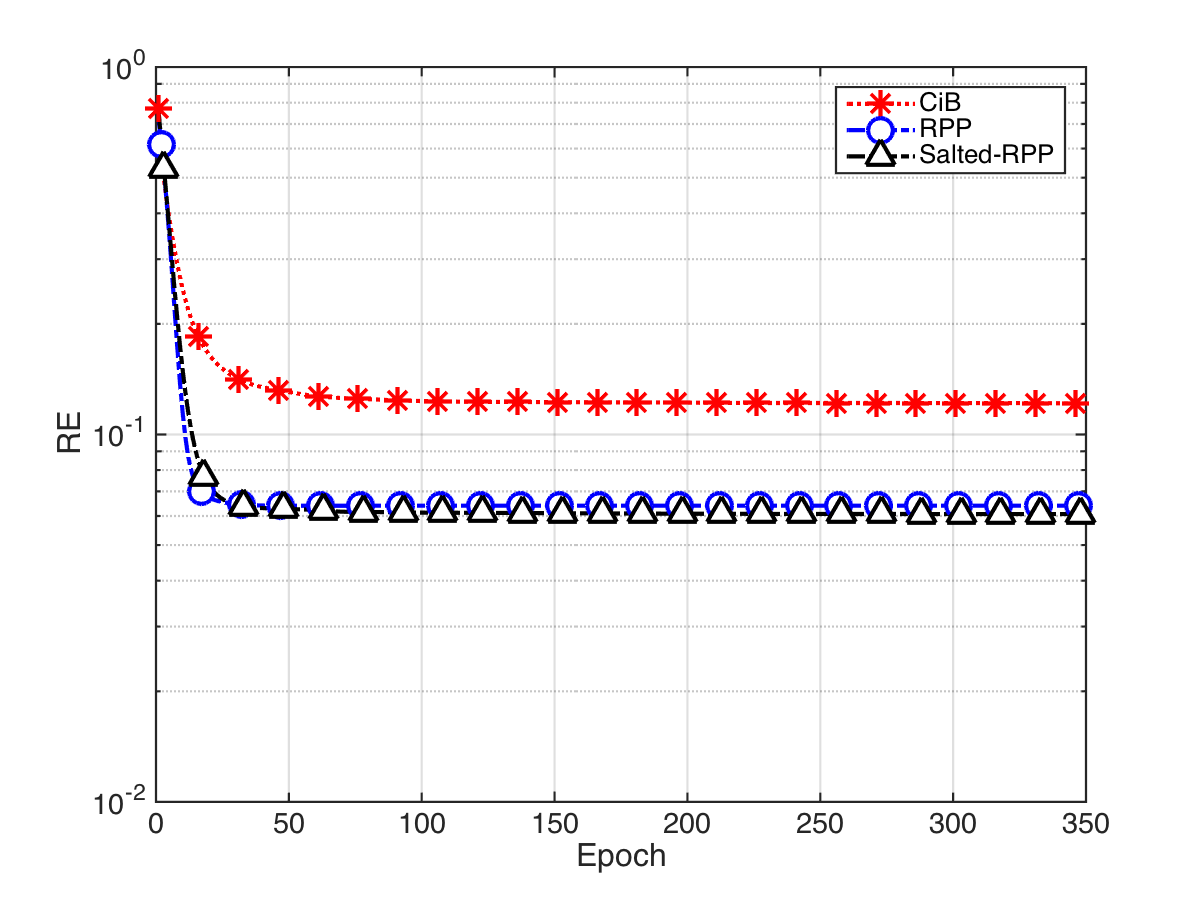}}
\subfigure[Gaussian-DRSAM with PPC(0.5) \commentout{w. PPC(-1,1,0.5)}]{\includegraphics[width=7cm]{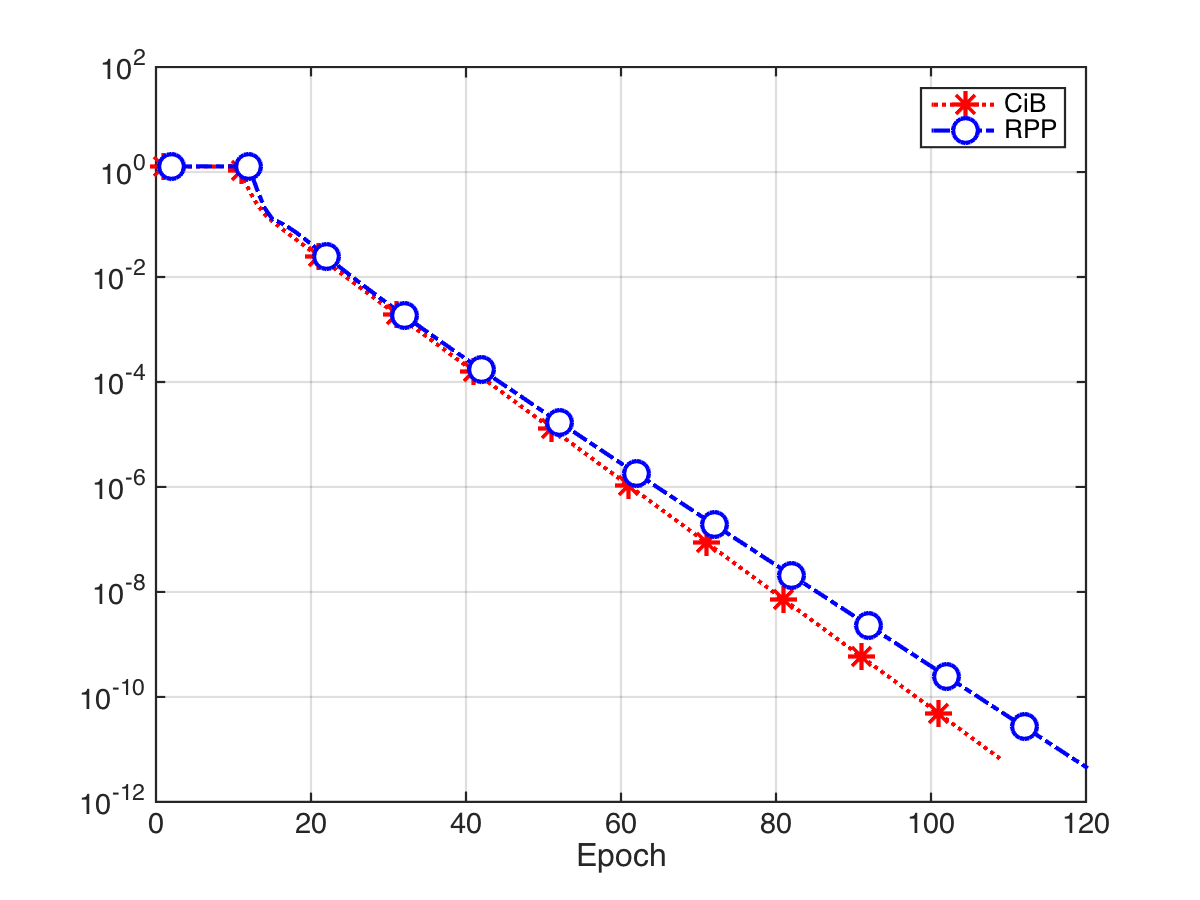}}
\caption{RE versus epoch for blind ptychography for various objects indicated in the legend by (a) rPIE and (b) DRSAM (RPP rate = 0.8015;
CiB rate=0.7787\commentout{with the linear phase factor $\mathbf{k}=(-1,1.0517)$; with linear phase factor $\mathbf{k}=(-1,1)$}).  }
 \label{fig:rPIE}
 \end{figure}

In this section, we compare the performance of DRSAM in Fig. \ref{fig:bright} (a) with that of the regularized PIE (rPIE) \cite{rPIE17}, the most up-to-date version of ptychographic iterative engine (PIE).

Instead of using all the 64 diffraction patterns simultaneously to update the object and probe estimates, rPIE uses
one diffraction pattern at a time in a random order. As such rPIE is analogous to minibatch gradient descent in machine learning. The potential benefits  include efficient memory use and a good speed boost by parallel computing resources. Unfortunately, rPIE
often fails to converge  in the current setting. 

To obtain reasonable results for rPIE, we make two adjustments. First, we reduce the phase range of RPP from 
$(-\pi,\pi]$ to $(-{\pi}/{2}, {\pi}/{2}$] which is an easier object to reconstruct. Second,  for rPIE we use PPC(0.025) for the probe initialization  which restricts the probe phase uncertainty to $(-0.025\pi, 0.025\pi]$ instead of $(-\pi/2,\pi/2]$. 

There are three adjustable parameters in rPIE and we select these values $\alpha =0.95, \gamma_{\rm prb}=0.95, \gamma_{\rm obj}=0.9$ (see \cite{rPIE17} for definition). The order of updating small patches is randomly shuffled in each experiment. For each test image, we run 20 independent experiments and present the best run in Fig. \ref{fig:rPIE}. For ease of comparison, Figure \ref{fig:rPIE}(b) shows the corresponding 
results by Gaussian-DRSAM with $\rho=1$.  

 \subsection{Poisson noise}
\begin{figure}[t]
\centering
\includegraphics[width=10cm]{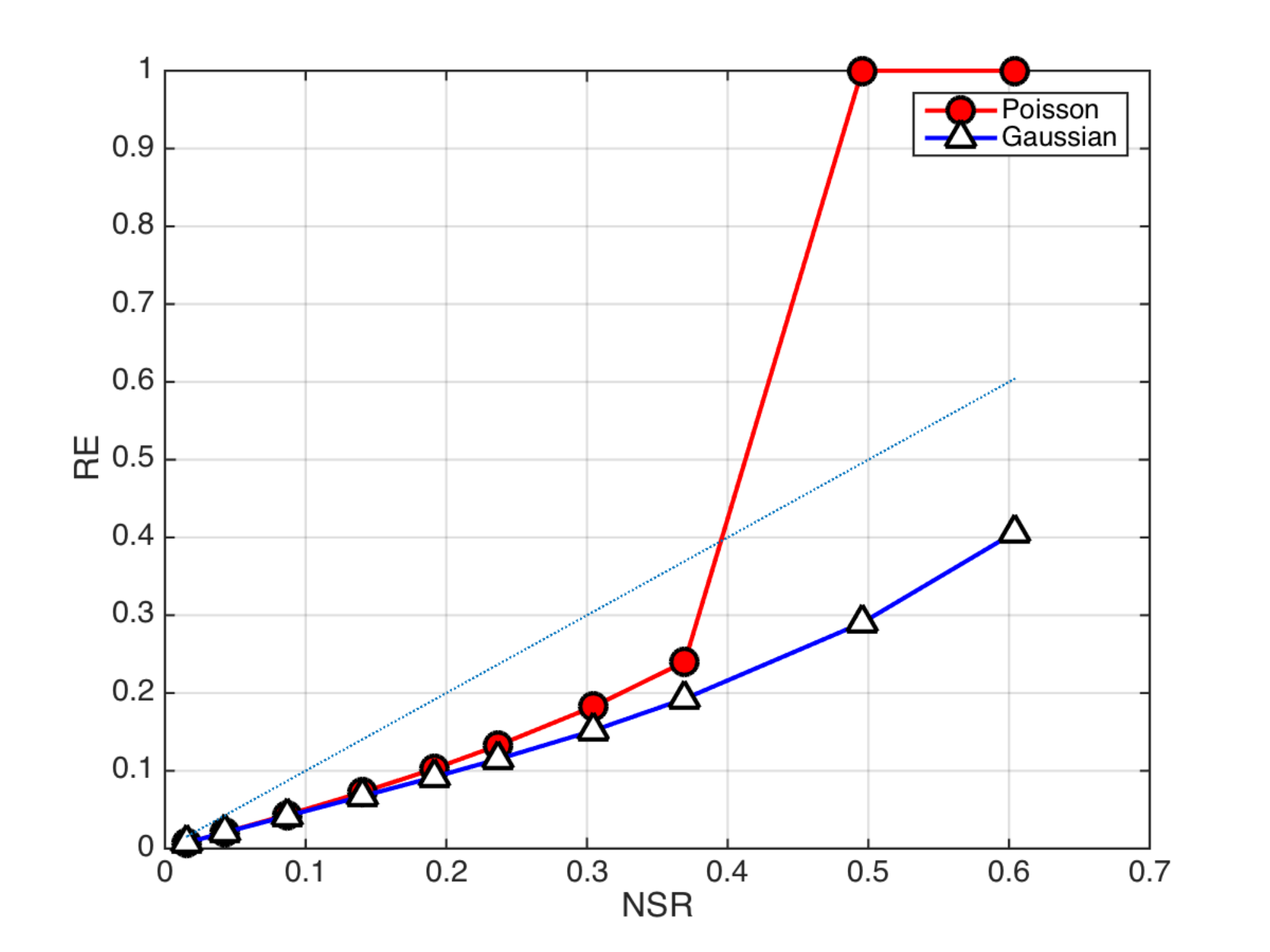}
\caption{RE versus NSR for reconstruction of CiB.}
\label{fig:noise}
\end{figure}
For noisy measurement, the level of noise is measured in terms of the noise-to-signal ratio (NSR).
\begin{equation*}
{\rm NSR}= \frac{\|b-|Af|\|}{\|Af\|}
\end{equation*}
where $A$ is the true measurement matrix and $f$ the true object. Because the noise dimension $N$ is roughly 16 times that of the object dimension,
the feasibility problem is  inconsistent with high probability.  

Figure \ref{fig:noise} shows RE versus  NSR for CiB  by Poisson-DRS and Gaussian-DRS with the periodic boundary condition, i.i.d. probe and the full-rank scheme. The maximum number of epoch in DRSAM is limited to $100$. The RR stabilizes  usually  after 30 epochs. The (blue) reference straight line has slope = 1. We see that the Gaussian-DRS outperforms the  Poisson-DRS, especially when  the Poisson RE becomes unstable for NSR $\ge 35\%$.  As noted in \cite{rPIE17,adaptive,AP-phasing} fast convergence (with the Poisson log-likelihood function) may introduce noisy artifacts and reduce reconstruction quality. 

Most important, Figure \ref{fig:noise} confirms that though provably non-convergent in the inconsistent case, Gaussian-DRSAM with $\rho=1$  can yield
reasonable solutions under practical termination rules.

\section{Conclusion and discussion}\label{sec:last}

We have presented and performed fixed point analysis for DRS methods of phase retrieval and ptychography based on
the proximal relaxation of AAR with the relaxation parameter $\rho$. 

For Gaussian-DRS, we have proved that
for $\rho\ge 1$ all   attracting fixed points must be regular solutions (Theorem \ref{thm:stable})
and that for $\rho\ge 0$ all regular solutions are   attracting (Theorem \ref{thm:stable2}). 
In other words, for $\rho\ge 1$, the problem of stagnation near a non-solutional  fixed point,  a common problem with AP, is precluded. 
On the other hand,
the problem of divergence (associated with AAR) in the inconsistent case does not arise in view of Theorem  \ref{thm:bounded}. 

In addition, we have given an explicit formula for the optimal parameter $\rho_*$ and the optimal rate of convergence in terms of
the spectral gap  (Corollary \ref{cor1}). 

When applied to standard phase retrieval with two coded diffraction patterns, Gaussian-DRS converges geometrically from random initialization.
When applied to blind ptychography, DRSAM,  even with a far from optimal step size,  converges geometrically under  the nearly minimum conditions established in the uniqueness theory \cite{blind-ptych}. 
Our Python codes are posted on {\tt https://github.com/AnotherdayBeaux/Blind$_{-}$Ptychography$_{-}$GUI}. 
 
The holy grail of optimization approach has been finding a globally convergent algorithm whose underlying attractors are fixed points.  It is worthwhile then to reflect on our results from the  global convergence perspective of \cite{LP}. 

We have already pointed out that the analysis in \cite{LP} is not applicable to non-differentiable loss functions. As discussed in Section \ref{sec:stepsize}, this   technical issue
has a profound effect on the convergence behavior in the inconsistent case: Gaussian-DRS with $\rho\ge 1$ does not converge, globally or locally. This is an unexpected consequence of Theorem \ref{thm:stable}. 

Our numerical experiments with noisy data, however, suggest that non-convergent DRS sequences are nevertheless well-behaved (probably due to hitherto unknown well-controlled attractors) and produce noise-amplification factor of about $\half$ when terminated. Analysis of such (possibly strange) attractors and their impacts on numerics  is an interesting topic for future research and at the frontier of numerical analysis. 

Moreover, the global convergence framework is typically  based on the construction of a non-increasing merit function along the iterated sequence (i.e. Lyapunov-like function)
that  requires the step size (reciprocal of $\rho$) to be sufficiently small, resulting in slow convergence in practice.  

Nice as it is,   perhaps algorithmic convergence should not be our fixation in the case of noisy data. It may be more useful, for numerical purposes, 
to solve noisy phase retrieval problem by algorithms with non-trivial (non-point-like) attractors which are necessarily non-convergent in the traditional sense.

\appendix
\section{Measurement matrices}\label{app:matrix}

Let $\IZ_n^2=\lb 0,n-1\rb^2$ be the object domain containing the support of the discrete object $f$ where $\lb k, l\rb$ denotes the integers between, and including,  $k\le l\in \IZ$.  Let $\cM^{0}:=\IZ_m^2, m<n,$ be the initial probe area, i.e.  the support of the probe $\mu^{0}$ describing the illumination field. 

Let $\cT$ be the set of all shifts, including $(0,0)$,  involved  in the ptychographic measurement. 
 Denote by $\mu^\bt$ the $\bt$-shifted probe for all $\bt\in \cT$ and $\cM^\bt$ the domain of
$\mu^\bt$. Let $f^\bt$ the object restricted to $\cM^\bt$.
We  refer to each $f^\bt$ as a part of $f$ and write $f=\vee_\bt f^\bt$ {  where $\vee$ is the ``union" of functions consistent over their common support set}. In ptychography, the original object is broken up into a set of overlapping object parts, each of which produces a $\mu^\bt$-coded diffraction pattern.  
The totality of the coded diffraction patterns is called the ptychographic measurement data.   For convenience, we assume the value zero for $\mu^\bt, f^\bt$ outside of $\cM^\bt$
 and the periodic boundary condition on $\IZ_n^2$ when $\mu^\bt$ crosses over the boundary of $\IZ_n^2$.

Let the $\mu$-Fourier transform of $f^{0}$ be written as 
\[
F^0(\bw)=\sum_{\bk\in \cM^0} e^{-\im 2\pi \bk\cdot\bw} \mu^{0}(\bk) f^{0}(\bk),\quad \bw=(w_1,w_2)\in [0,1]^2. 
\]
and the $\mu$-coded diffraction pattern as  \beq
|F^0(\bw)|^2= \sum_{\bk \in \widetilde\cM^0}\lt\{\sum_{\bk'\in \cM^0} \mu^{0}(\bk+\bk')f^{0}(\bk'+\bk)\overline{\mu^{0}(\bk')f^{0}(\bk')}\rt\}
   e^{-\im 2\pi \bk\cdot \bom}\label{auto}
   \eeq
   where
    \begin{equation*}
 \widetilde \cM^{0} = \{ (k_1,k_2)\in \IZ^2: -m+1\le k_1 \le m-1, -m+1\le k_2\leq m-1 \}.  
 \end{equation*}
 Here and below the over-line notation means
complex conjugacy. In view of \eqref{L}, we sample the coded diffraction pattern 
 on the grid 
\beq\label{L}
 L = \Big\{(w_1,w_2)\ | \ w_j = 0,\frac{1}{2 m-1},\frac{2}{2m-1},\ldots,\frac{2m-2}{2m-1}\Big\}. 
\eeq

We assume  {randomness} in the phases $\theta$ of the mask function 
$
\mu^0(\bn)=|\mu^0|(\bn)e^{\im \theta(\bn)}
$
where  $\theta(\bn)$ are independent, continuous real-valued random variables over $[0,2\pi)$. 
We also require that  $|\mu^0|(\bn)\neq 0,\forall \bn\in \cM^0$. 

 \begin{figure}[t]
\begin{center}
\subfigure[Matrix $A_\nu$]{
\includegraphics[width=8cm]{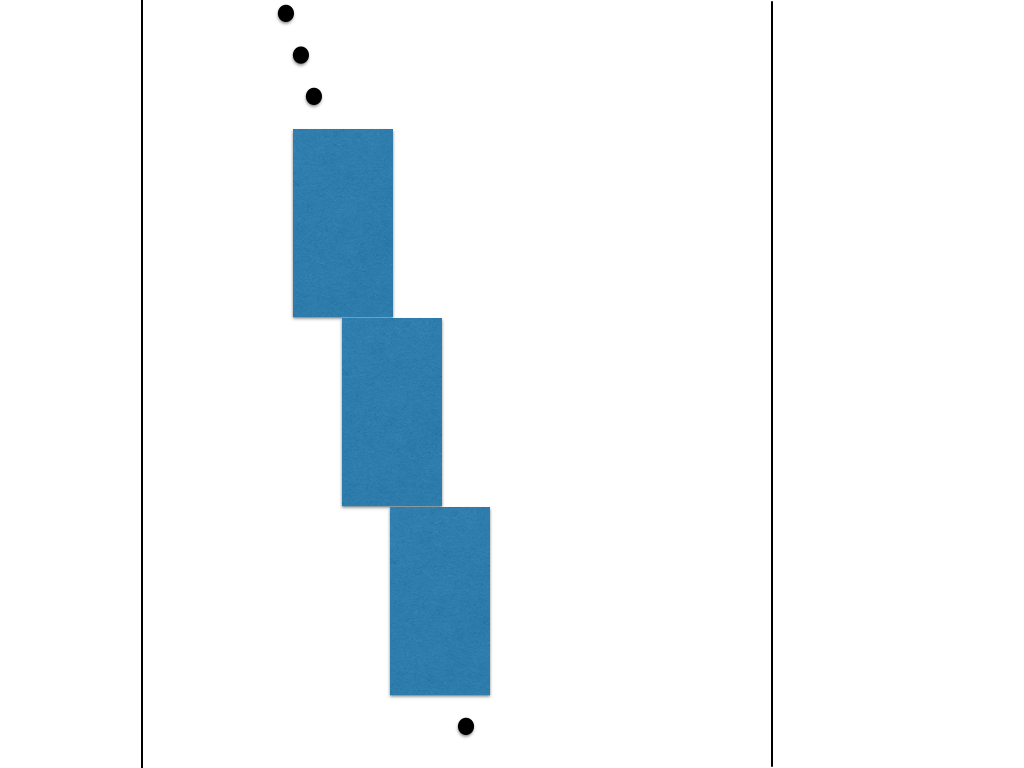}
}\hspace{-2cm}
\subfigure[Matrix $B_g$]{
\includegraphics[width=8cm]{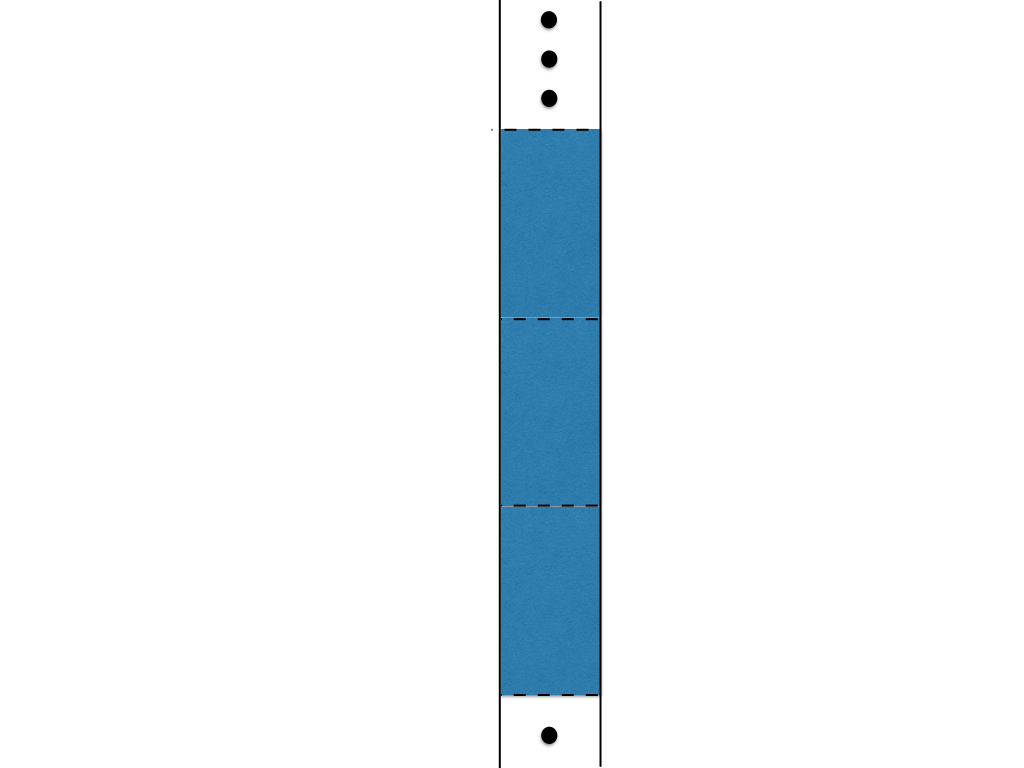}
}
\caption{(a) $A_\nu$ is a concatenation of shifted
blocks $\{\Phi\,\diag(\nu^\bt):\,\,\bt\in \cT\}$; (b) $B_g$  is a concatenation of unshifted 
blocks $\{\Phi\,\diag(g^\bt):\,\,\bt\in \cT\}$.   In both cases, each block gives rise to
a coded diffraction pattern $|\Phi (\nu^\bt\odot g^\bt)|$.
}
\label{fig10}
\end{center}
\end{figure}

Let $\cF(\nu^0,g)$ be the bilinear transformation representing the totality of the Fourier (magnitude and phase) data for any probe $\nu$ and object $g$.
From $\cF(\nu^0,g)$ we can define two measurement matrices. First, for a given $\nu^0\in \IC^{m^2}$, 
let $A_\nu$ be defined via the relation $A_v g:=\cF(\nu^0,g)$ for all $g\in \IC^{n^2}$; second, for a given $g\in \IC^{n^2}$,
let $B_g$ be defined via $B_g \nu=\cF(\nu^0,g)$ for all $\nu^0\in \IC^{m^2}$. 

More specifically, let $\Phi$ denote the $ L$-sampled Fourier matrix. The measurement matrix $A_\nu$ is a concatenation of $\{\Phi \,\diag(\nu^\bt):\bt \in \cT\}$ (Figure \eqref{fig10}(a)). Likewise,
$B_g$ is $\{\Phi \,\diag(g^\bt):\bt \in \cT\}$ stacked on top of each other (Figure \eqref{fig10}(b)). 
Since
$\Phi$ has orthogonal columns, both $A_\nu$ and $B_g$ have orthogonal columns and their pseudo-inverses are efficient to compute. 

We simplify the notation by setting $A=A_\mu$ and $B=B_f$.  

\section{The Poisson versus Gaussian log-likelihood functions}\label{app:likelihood}

Consider the Poisson distribution
\[
P(n)={\lambda^ne^{-\lambda}\over n!}
\]
Let $n=\lambda(1+\ep)$ where $\lamb\gg 1$ and $\ep\ll 1$.
Using  Stirling's formula 
\[
n!\sim \sqrt{2\pi n} e^{-n} n^n
\]
in the Poisson distribution, we obtain
\beqn
P(n)&\sim &{\lamb^{\lamb(1+\ep)}e^{-\lamb}\over \sqrt{2\pi} e^{-\lamb(1+\ep)} [\lamb(1+\ep)]^{\lamb(1+\ep)+1/2}}\\
&\sim &{1\over \sqrt{2\pi\lamb} e^{-\lamb\ep} (1+\ep)^{\lamb(1+\ep)+1/2}}.
\eeqn
By the asymptotic 
\[
(1+\ep)^{\lamb(1+\ep)+1/2}\sim e^{\lamb\ep+\lamb\ep^2/2}
\]
we have
\beq
\label{PG1}
P(n)\sim {e^{-\lamb\ep^2/2}\over \sqrt{2\pi \lamb}}={e^{-(n-\lambda)^2/(2\lamb)}\over \sqrt{2\pi\lamb}}. 
\eeq
Namely  in the low noise limit the Poisson noise is equivalent to the Gaussian noise of the mean $|Af|^2$ and
 the variance equal to the intensity of the diffraction pattern.  
\commentout{and get the pdf
\[
 {e^{-\xi^2/(2\sigma^2)}\over \sqrt{2\pi\sigma^2}}
\]
}
The overall SNR can be tuned by varying the  signal energy $\|Af\|^2$. 

The negative log-likelihood function for the right hand side of  \eqref{PG1} is 
\beq
\label{pg}
\sum_j \ln |u[j]| +{1\over 2} \lt|{b[j]\over |u[j]|}-|u[j]|\rt|^2,\quad b= \mbox{noisy diffraction pattern.}
\eeq
For small NSR and in the vicinity of $b$, we make the substitution 
\[
{\sqrt{b[j]}\over |u[j]|}\to 1,\quad \ln|u[j]|\to \ln\sqrt{b[j]}
\]
to obtain 
\beq
\label{pg2}
\mbox{const.}+ {1\over 2} \sum_j \lt|\sqrt{b[j]}-|u[j]|\rt|^2.
\eeq

\section{Equivalence between DRS and ADMM}\label{sec:noisy-admm}

We show that  ADMM applied to the augmented Lagrangian 
\beq
\label{AL2}
\cL(y,z)=K(y)+L(z)+\lamb^*(z-y)+{\rho\over 2}\|z-y\|^2
\eeq
in the order
 alternatively  as
\beq
\label{701'}z_{k+1}&=& \arg\min_z\cL(y_{k+1}, z, \lamb_k)\\
\label{700'}y_{k+1}&=&\arg\min_x\cL(y,z_k,\lamb_k)\\
\label{702'}\lamb_{k+1}&=&\lamb_k+\rho (z_{k+1}-y_{k+1}). 
\eeq
is equivalent to DRS. 

Let
\beq
\label{700"}z_{k+1}&=& \arg\min_z\cL(y_{k}, z, \lamb_k)=\prox_{L/\rho}(y_k-\lamb_k/\rho)\\
\label{701"}y_{k+1}&=&\arg\min_x\cL(y,z_{k+1},\lamb_k)=\prox_{K/\rho}(z_{k+1}+\lamb_k/\rho)
\eeq
and consider the new variable
\[
u_k:=z_k+\lamb_{k-1}/\rho.
\]
We have from \eqref{702'} that
\beqn
u_{k+1}=y_{k+1}+\lamb_{k+1}/\rho.
\eeqn
By \eqref{701"}, we also have
\[
y_{k+1}=P_X (z_{k+1}+\lamb_k/\rho)=P_X u_{k+1}
\]
and 
\beqn
y_k - \lamb_k/\rho= 2y_{k}-u_k=R_X u_{k}. 
\eeqn
So  \eqref{700"} becomes
\beq
z_{k+1}=\prox_{L/\rho}(R_X u_k)\label{710}.
\eeq
Note also that by \eqref{702'}
\beqn
u_k-P_X u_k=\lamb_k/\rho 
\eeqn
and hence
\beqn
u_{k+1}=z_{k+1}+\lamb_k/\rho=u_k-P_X u_k+\prox_{L/\rho}(R_X u_k)
\eeqn
which is exactly the DRS scheme \eqref{G1} after rearrangement.

\section{Perturbation analysis of Poisson-DRS}\label{app:Poisson}
The full analysis of the  Poisson-DRS \eqref{P1} is more challenging. Instead, we give a
perturbative derivation of  analogous result to Theorem \ref{thm:bounded} for the Poisson-DRS
with small positive $\rho$.

For small $\rho$, by keeping only the terms up to $\cO(\rho)$ we obtain
the perturbed DRS:
\beq
u_{k+1}=\half u_k-\half(1-{\rho\over 2}) R_X u_k +P_Y R_X u_k.
\eeq

Writing 
\[
I=P_X+P^\perp_X\quad\mbox{and}\quad R_X=P_X-P^\perp_X, 
\]
we then have the estimates
\beqn
 \|u_{k+1}\|
&\le& \|{\rho\over 4} P_X u_k+(1-{\rho\over 4})P^\perp_X u_k\|+\|P_Y R_X u_k\|\nn\\
&\le &(1-{\rho\over 4})\|u_k\|+\|b\| 
\eeqn
since $\rho$ is small. 
Iterating this bound, we obtain
\beqn
\|u_{k+1}\|\le (1-{\rho\over 4})^k\|u_1\|+\|b\|\sum_{j=0}^{k-1}(1-{\rho\over 4})^j
\eeqn
and hence
\beq
\label{881}
\limsup_{k\to\infty} \|u_k\|\le {4\over \rho} \|b\|. 
\eeq
Note that the small $\rho$ limit and the Poisson-to-Gaussian limit in Appendix \ref{app:likelihood}
do not commune, resulting in a different constant in \eqref{881} from Theorem \ref{thm:bounded}.

\section{Eigen-structure}\label{app:eigen}

The vector space $\IC^N=\IR^N\oplus_\IR i\IR^N$ is 
isomorphic to
$\IR^{2N}$ via the map 
\[ V(v):=\left[
\begin{array}{c}
\Re(v)     \\
 \Im(v)  
\end{array}
\right],\quad \forall v \in \IC^{N}\] 
and endowed with the real inner product
\[
\langle u, v\rangle :=\Re(u^*v)=V(u)^\top V(v),\quad u,v\in \IC^N.
\]
We have
\beq\label{53}
V(H\xi)= \lt[\begin{matrix}
\Re[H] \Re[\xi]+
\Im[H] \Im[\xi]\\
\Re[H] \Im[\xi]-
\Im[H] \Re[\xi]
           \end{matrix}\rt] 
=\lt[\begin{matrix}
\cH^\top V(\xi)\\
           \cH^\top V(-\im \xi)
           \end{matrix}\rt],\quad \xi\in \IC^n. 
           \eeq

 Let $\lambda_1\ge \lambda_2\ge \cdots\ge \lambda_{2n}\ge \lambda_{2n+1}=\cdots=\lambda_{N}=0$ be the singular values of $\cH$ in \eqref{Bv} with the corresponding right singular vectors
$\{\eta_k\in \IR^{N}\}_{k=1}^{N} $ and left singular vectors $\{{\xi}_k\in \IR^{2n}\}_{k=1}^{2n}$.
By definition, for  $k=1,\ldots,2n$, 
\beq
\label{53'}
H^* \eta_k&=&\lambda_k G^{-1} (\xi_k),\\
\Re[H G^{-1}( \xi_k)]&=& \lambda_k \eta_k. 
\eeq 
\begin{prop}\label{B*bound} 
We have $\xi_{1}=V(f)$,  $\xi_{2n^2}=V(-\im f)$, $\lambda_1=1, \lambda_{2n^2}=0$ as well
as 
  $\eta_{1}= |Af|$.
\end{prop}
\begin{proof} Since
 \beq\nn
&& H f=\Om^* A f=|Af|
\eeq
we have by  \eqref{53}
\beq
& &\Re[H f]=\cH^\top \xi_{1}=|Af|, \quad \Im[H f]=\cH^\top \xi_{2n^2}=0  \label{56}\eeq
 and hence the results. 
\end{proof}

\begin{cor}\label{cor5.5}
 \beq 
\label{63} \lambda_2&= &\max \{\|\Im(H u)\|: {u\in \IC^n, u\perp \im f,\|u\|=1} \}\\
& =& \max \{\|\cH^\top u\|: {u\in \IR^{2n^2},  u\perp \xi_1, \|u\|=1} \} \nn
 \eeq
\end{cor}
\begin{proof}
By \eqref{53}, 
\[
\Im[H u]=\cH^\top V(-\im u).
\]
The orthogonality condition $\im u\perp f$ is equivalent to
\[
V(x_0)\perp V(-\im u).
\]
Hence, by 
 Proposition \ref{B*bound} $\xi_2 $ is the maximizer of the right hand side of \eqref{63}, yielding the desired value $\lambda_2$. 

\end{proof}

 \begin{prop}\label{Buu}   For $k=1,\ldots, 2n^2$, 
  \beq\label{57}
  \lambda_k^2+\lambda_{2n^2+1-k}^2=1
  \eeq
\beq
\label{58}
\xi_{2n^2+1-k}&=&V( -\im V^{-1}(\xi_k) )\\
\xi_{k}&=&V(\im V^{-1}(\xi_{2n^2+1-k}) ).\label{59}
\eeq
\end{prop}
\begin{proof}
Since $H$ is an isometry, we have $\|w\|=\|H w\|,\forall w\in \IC^n$. On the other hand, 
we have
\beqn
\|Hw\|^2=\|V(Hw)\|^2=\|\cH^\top V(w)\|^2+\|\cH^\top V(-\im w)\|^2\ldots
\eeqn
and hence
\beq
\|V(w)\|^2=\|\cH^\top V(w)\|^2+\|\cH^\top V(-\im w)\|^2. \label{Key}
\eeq

Now we prove \eqref{57}, \eqref{58} and \eqref{59} by induction. 

Recall the variational characterization of the singular values/vectors
\beq \label{60'}
\lambda_j=\max\| \cH^\top {u}\|,& \xi_j=\hbox{\rm arg}\max
 \| \cH^\top  {u}\|, & \hbox{s.t.}\,\, {u}\perp {\xi}_1,\ldots,  {\xi}_{j-1},\quad \|u\|=1
 \eeq
By Proposition \ref{B*bound}, \eqref{57}, \eqref{58} and \eqref{59} hold for $k=1$. 
Suppose  \eqref{57}, \eqref{58} and \eqref{59} hold for $k=1,\ldots,j-1$ and 
we now show that they also hold for $k=j$. 

 Hence by \eqref{Key}
 \[
 \lambda^2_j=\max_{\|u\|=1}\| \cH^\top {u}\|^2=1-\min_{\|v\|=1}\| \cH^\top {v}\|^2,\quad  \hbox{s.t.}\,\, {u}\perp {\xi}_1,\ldots,  {\xi}_{j-1},\quad v=V(-\im V^{-1}(u)).
 \]
 The condition $ {u}\perp {\xi}_1,\ldots,  {\xi}_{j-1}$ implies $v\perp {\xi}_{2n^2},\ldots,  {\xi}_{2n^2+2-j}$ and vice versa. By the dual variational characterization to \eqref{60'}
 \beqn
 \lambda_{2n^2+1-j}=\min\| \cH^\top {u}\|,& \xi_{2n^2+1-j}=\hbox{\rm arg}\min
 \| \cH^\top  {u}\|, & \hbox{s.t.}\,\, {u}\perp {\xi}_{2n^2},\ldots,  {\xi}_{2n^2+2-j}, \|u\|=1,
\eeqn 
 we have  
 \[
 \lambda_j^2=1-\lambda_{2n^2+1-j}^2,\quad 
\xi_{2n^2+1-j}=V(-\im V^{-1}(\xi_j)).
\] 

\end{proof}

\begin{prop}\label{Srate}
For each $k=1,\ldots, 2n^2$, 
\begin{eqnarray}\label{B5}
&&HH^*\eta_k=\lambda_k(\lambda_k \eta_k+\im\lambda_{2n^2+1-k} \eta_{2n^2+1-k}),\\
&&HH^*\eta_{2n^2+1-k}=\lambda_{2n^2+1-k}(\lambda_{2n^2+1-k} \eta_{2n^2+1-k}-\im\lambda_k \eta_k)\label{B51}
\end{eqnarray}
implying 
\beqn
HH^*&=&\lt[ \begin{matrix}
\lambda_k^2& \lambda_k\lambda_{2n^2+1-k}\\
\lambda_k\lambda_{2n^2+1-k}&\lambda_{2n^2+1-k}^2
\end{matrix}\rt]
\eeqn
in  the basis of $\eta_k, \im \eta_{2n^2+1-k}$.
\end{prop}

\begin{proof} By definition, $\cH\eta_k= \lambda_k {\xi}_k.$ Hence 
\[H^* \eta_k=(\Re[H^*]+\im\Im[H^*]) \eta_k=\lambda_k (\xi_k^{\rm R}+\im\xi_k^{\rm I})
\]
where 
\[
\xi_k=\lt[\begin{matrix} \xi_k^{\rm R}\\
\xi_k^{\rm I}
\end{matrix}
\rt],\quad \xi_k^{\rm R},\xi_k^{\rm I}\in \IR^n.
\]

On the other hand, $\cH^\top \xi_k=\lambda_k \eta_k$ and hence 
\beq
\label{59'}
\Re[H]\xi_k^{\rm R}-\Im[H] \xi_k^{\rm I}=\lambda_k \eta_k. 
\eeq

Now we compute $HH^* \eta_k$ as follows.
\beq
HH^* \eta_k&=& \lambda_k H(\xi_k^{\rm R}+\im\xi_k^{\rm I})\label{60}\\
&=& \lambda_k (\Re[H]+\im\Im[H])(\xi_k^{\rm R}+\im\xi_k^{\rm I})\nn\\
&=&\lambda_k( \Re[H]\xi_k^{\rm R}-\Im[H]\xi_k^{\rm I})+\im\lambda_k (\Re[H]\xi_k^{\rm I}+\Im[H] \xi_k^{\rm R})\nn\\
&=&\lambda_k^2\eta_k+\im\lambda_k (\Re[H]\xi_k^{\rm I}+\Im[H] \xi_k^{\rm R})\nn
\eeq
by \eqref{59'}.

Notice that 
\beq
 \label{61}\Re(H)\xi_k^{\rm I}+\Im(H) \xi_k^{\rm R}&=&\cH^\top \lt[
\begin{matrix}
\Re(-\im V^{-1}(\xi_k))\\
\Im(-\im V^{-1}(\xi_k))\end{matrix}\rt]\\
&=&\cH^\top V(-\im V^{-1}(\xi_k))\nn\\
&=&\cH^\top \xi_{2n^2+1-k}\nn\\
&=&\lambda_{2n^2+1-k} \eta_{2n^2+1-k}\nn
\eeq
by Proposition \ref{Buu}. 

Putting \eqref{60} and \eqref{61} together, we have
\eqref{B5}. Likewise, \eqref{B51} follows from a similar calculation. 
\end{proof}

\section*{Acknowledgment} This research is supported by  the US National Science Foundation  grant DMS-1413373 and
SIMONS FDN 2019-24.  A.F. thanks National Center for
Theoretical Sciences (NCTS), Taiwan,   where  the present work was carried out, for the hospitality  during his visits  in June and August 2018.   

\bibliographystyle{siamplain}
\bibliography{references}

\end{document}